\documentclass[11pt, a4paper]{amsart}
\usepackage[euler-digits]{eulervm}

\usepackage{amsfonts, amsthm, amssymb, amsmath, stmaryrd}
\usepackage{mathrsfs,array}
\usepackage{eucal,fullpage,times,color,enumerate,accents}
\usepackage{url}

\usepackage{color}
\usepackage{mathrsfs}
\usepackage{amssymb}
\usepackage{tikz}
\usepackage{tikz-cd}
\usepackage{bm}
\usepackage{enumerate}
\usetikzlibrary{calc}

\title{Tropicalizing the Space of Admissible Covers}
\author{Renzo Cavalieri, Hannah Markwig, and Dhruv Ranganathan}
\date{\today}
\address{Department of Mathematics, Colorado State University}
\email{renzo@math.colostate.edu}

\address{Fachrichtung Mathematik, Universit\"at des Saarlandes}
\email{hannah@math.uni-sb.de}

\address{Department of Mathematics, Yale University}
\email{dhruv.ranganathan@yale.edu}

\newtheorem{theorem}{Theorem}
\newtheorem{corollary}[theorem]{Corollary}
\newtheorem{lemma}[theorem]{Lemma}
\newtheorem{proposition}[theorem]{Proposition}
\newtheorem{definition}[theorem]{Definition}

\newenvironment{remark}{\begin{rem1}\em}{\end{rem1}}
\newenvironment{warning}{\begin{war1}\em}{\end{war1}}
\newenvironment{example}{\begin{exa1}\em}{\end{exa1}}
\newenvironment{convention}{\begin{conv1}\em}{\end{conv1}}
\newtheorem{rem1}[theorem]{Remark}
\newtheorem{war1}[theorem]{Warning}
\newtheorem{exa1}[theorem]{Example}
\newtheorem{conv1}[theorem]{Convention}

\newcommand{\A}{{\mathbb{A}}}           
\newcommand{\CC} {{\mathbb C}}

\newcommand{\PP}{\mathbb{P}}         
		
\newcommand{\RR} {{\mathbb R}}		
\newcommand{\ZZ} {{\mathbb Z}}

\DeclareMathOperator{\spec}{Spec}

\newcommand{\M}{\overline{{M}}}

\newcommand{\cal}{\mathcal}

\def\cH{{\cal H}}

\def\cM{{\cal M}}

\def\cO{{\cal O}}

\newcommand{\Mbar}{\overline{\cM}}
\newcommand{\Hbar}{\overline{\cH}}
\newcommand{\HMbar}{\overline{\cH\cM}}

\begin{document}

\pagestyle{plain}
\maketitle

\begin{abstract}
We study the relationship between tropical and classical Hurwitz moduli spaces. Following recent work of Abramovich, Caporaso and Payne, we outline a tropicalization for the moduli space of generalized Hurwitz covers of an arbitrary genus curve. Our approach is to appeal to the geometry of  admissible covers, which compactify the Hurwitz scheme. We study the relationship between a combinatorial moduli space of tropical admissible covers and the skeleton of the Berkovich analytification of the classical space of admissible covers. We use techniques from non-archimedean geometry to show that the tropical and classical tautological maps are compatible via tropicalization, and that the degree of the classical branch map can be recovered from the tropical side. As a consequence, we obtain a proof, at the level of moduli spaces, of the equality of classical and tropical Hurwitz numbers.
\end{abstract}

\setcounter{tocdepth}{1}
\tableofcontents
\pagebreak

\section{Introduction}
The primary objective of this paper is to establish a geometric and functorial relationship between the moduli space of Hurwitz covers of an algebraic curve, and a combinatorial moduli space of tropical Hurwitz covers\footnote{The notion of a tropical Hurwitz cover for  discrete graphs was introduced by Caporaso in \cite{cap:gonality}, in order to study  the gonality of graphs.  In particular she characterizes which graphs covering a tree are dual graphs of a classical admissible cover. }. Such a relationship is given by a morphism of cone complexes from the skeleton of the Berkovich analytification of the space of admissible covers to the moduli space of tropical Hurwitz covers. We show that tropicalization commutes with the natural tautological source and branch maps on these moduli spaces. Consequently, we recover the equality between tropical and classical Hurwitz numbers, originally proved via combinatorial and topological methods in~\cite{BBM, CJM1}.

\subsection{Results}

Fix a vector of partitions $\vec \mu = (\mu^1,\ldots, \mu^r)$ of an integer $d>0$. Denote by  $\cH_{g\to h,d}(\vec \mu)$ the space of degree $d$ Hurwitz  covers $[D\to C]$  of smooth genus $h$ curves by genus $g$ curves with ramification $\mu^i$ over smooth marked points $p_i$ of $C$, and simple ramification over smooth marked points $q_1,\ldots, q_s$,  and by $\Hbar_{g\to h,d}(\vec \mu)$ its admissible cover compactification.  We denote by $\Hbar^{an}_{g\to h,d}(\vec\mu)$ (resp. $\Mbar^{an}_{g,n}$) the Berkovich analytification of the space of admissible covers (resp. the moduli space of stable curves). We  use  $\cH^{trop}_{g\to h,d}(\vec \mu)$ to denote the space of tropical admissible covers of genus $h$ tropical curves by genus $g$ tropical curves, with expansion factors along infinite edges prescribed by $\vec\mu$. See Sections~\ref{sec: comb-const}, \ref{sec: bkgrnd} for definitions and background. 

Our first result studies the relationship between the  set theoretic tropicalization map from $\cH^{an}_{g\to h,d}(\vec\mu)$ to $\cH^{trop}_{g\to h,d}(\vec\mu)$ (see Definition~\ref{def: trop}) and the canonical projection to the skeleton from the space $\cH^{an}_{g\to h,d}(\vec\mu)$.

\begin{theorem}~\label{thm: theorem1}
The set theoretic tropicalization map $trop: \cH^{an}_{g\to h, d}(\vec\mu)\to \cH^{trop}_{g\to h, d}(\vec\mu)$ factors through the canonical projection from the analytification to its skeleton $\Sigma(\Hbar^{an}_{g\to h,d}(\vec\mu))$,
\begin{equation}
\begin{tikzcd}
\cH^{an}_{g\to h, d}(\vec \mu) \arrow{rr}{trop} \arrow{dr}[swap]{\bm p_H} & &  \cH^{trop}_{g\to h,d}(\vec \mu) \\
& \Sigma(\Hbar^{an}_{g\to h, d}(\vec \mu)) \arrow{ur}[swap]{trop_{\Sigma}}. &
\end{tikzcd}
\label{diag:tropcomp}
\end{equation}
Furthermore the map $trop_{\Sigma}$ is a surjective face morphism of cone complexes, i.e. the restriction of $trop_\Sigma$ to any cone of $\Sigma(\Hbar^{an}_{g\to h,d}(\vec\mu))$ is an isomorphism onto a cone of the tropical moduli space $\cH^{trop}_{g\to h,d}(\vec\mu)$. The map $trop_\Sigma$ extends naturally and uniquely to the extended complexes $\overline{\Sigma}(\Hbar^{an}_{g\to h,d}(\vec\mu))\to \Hbar^{trop}_{g\to h,d}(\vec\mu)$. 
\end{theorem}
The map $trop$ depends on the choice of the admissible cover compactification even when restricted to the analytification of the Hurwitz space. Intuitively, one may think of a point in $\cH^{an}_{g\to h, d}(\vec\mu)$ as a family of smooth covers over a punctured disk. The tropicalization of such point is obtained by extending the family to an admissible cover and metrizing the dual graph of the central fibers by the valuations of the smoothing parameters of the nodes.

The Hurwitz moduli spaces and their compactifications contain information about  the enumerative geometry of target curves. The (classical)  Hurwitz number $h_{g\to h,d}(\vec\mu)$ counts the number of covers with the above invariants and a fixed branch divisor.
Note that the definition of $\cH^{trop}_{g\to h,d}(\vec\mu)$ involves the algebraic data of triple Hurwitz numbers, as in~\cite{BM13}, closely related to the Hurwitz existence problem (\cite[Section 2.2]{cap:gonality}). This discussed in Section~\ref{sec: bkgrnd}.

Our next result shows that the tropical moduli space (and skeleton)  contain enough information to recover Hurwitz numbers.


\begin{theorem}\label{thm:branchdegree}
Let $\sigma_\Gamma$ be any fixed top dimensional cone of the tropical moduli space $\mathcal{M}^{trop}_{h,r+s}$. Denote by $\sigma^{\Hbar}_\Theta\mapsto \sigma_\Gamma$ a  cone  in the moduli space $\cH^{trop}_{g,d}(\vec \mu)$ of combinatorial type $\Theta$ such that the base graph of $\Theta$  is equal to $\Gamma$. The restriction of the tropical branch map is a surjective morphism of cones with integral structure of the same dimension, and consequently has a dilation factor which we denote $d_\Theta(br)$.

Refer to Definition~\ref{wetheta} for the precise definition of the  \textit{weight} $\omega(\Theta)$. Informally, this weight is a product of ``local Hurwitz numbers'', expansion factors along bounded edges, and an automorphism factor.

Then the  Hurwitz number is equal to: 
\begin{equation}
h_{g\to h,d}(\vec\mu) = \sum_{\sigma^{\Hbar}_\Theta\mapsto \sigma_\Gamma} \omega(\Theta)\cdot d_\Theta(br^{trop}).
\label{eq:hurn}
\end{equation}
\end{theorem}


When $h = 0$, and the profile $\vec \mu$ contains only two non-simple ramification types, the above theorem fully determines these Hurwitz numbers (\textit{double} Hurwitz numbers) in terms of purely tropical computations. For general $\vec\mu$, the result gives an explicit decomposition formula for the Hurwitz number $h_{g\to h,d}(\vec\mu)$ in terms of \textit{triple Hurwitz numbers}, i.e. Hurwitz numbers with 3 non-simple profiles.

The  right hand side of \eqref{eq:hurn} coincides with the definition of tropical Hurwitz numbers in \cite{BBM}. Theorem~\ref{thm:branchdegree} then provides a geometric proof for the following correspondence theorem.

\begin{theorem}[Bertrand--Brugall\'e--Mikhalkin~\cite{BBM}]\label{thm-BBM}
 Classical and tropical Hurwitz numbers coincide, i.e.\ we have $h_{g\to h,d}(\vec\mu)=h^{trop}_{g\to h,d}(\vec\mu)$.
\end{theorem}

For us Theorem~\ref{thm:branchdegree}  is a consequence of the functoriality of the tropicalization map. More precisely, the  Hurwitz number arises geometrically as the degree of a tautological map called the \textit{branch map}, which takes  a cover to its base curve, marked at its branch points.  The branch degree is computed in Section~\ref{subsec-BBM} by a computation on polyhedral domains in analytifications of formal tori. This provides a conceptual explanation for the determinantal formulas for tropical multiplicities obtained in~\cite{ CJM1}, and related works. In particular, the result seems adaptable to more general moduli spaces of maps.

We also study the source map, taking a cover to its source curve, marked at the entire inverse image of the branch locus. We have the diagram:
\[
\begin{tikzcd}
\Hbar_{g\to h, d}(\vec \mu) \arrow{r}{src} \arrow{d}[swap]{br} & \Mbar_{g,n}\\
\Mbar_{h,r+s}
\end{tikzcd}
\]
where $s$ is the number of simple branch points. There are analogous tautological morphisms on spaces of admissible covers of tropical curves.
Tropicalization is compatible with these two  tautological morphisms to the moduli space of curves in the following sense.

\begin{theorem}~\label{thm: theorem2}
Let $br$ denote the branch map $\Hbar^{}_{g\to h, d}(\vec \mu)\to \Mbar_{h,r+s}$, and $src$ denote the source map $\Hbar^{}_{g\to h, d}(\vec \mu)\to \Mbar_{g,n}$, where $n$ is the number of smooth points in the inverse image of the branch locus. Then the following diagram is commutative:

\[
\begin{tikzcd}
\Hbar^{an}_{g\to h, d}(\vec \mu) \arrow{dr}{trop} \arrow[swap]{dd}{br^{an}}  \arrow{rr}{src^{an}}  &   &   \Mbar^{an}_{g,n}  \arrow{d}{trop}\\
    & \Hbar^{trop}_{g,d}(\vec \mu) \arrow{d}{br^{trop}}  \arrow{r}{src^{trop}}  &  \Mbar^{trop}_{g,n}\\
\Mbar^{an}_{h,r+s} \arrow[swap]{r}{trop} & \Mbar^{trop}_{h,r+s}, & 
\end{tikzcd}
\]

The induced map on skeleta of the branch (resp. source) morphism factors as a composition of the map $trop_\Sigma$ to $\overline{\Sigma}(\Hbar^{an}_{g\to h,d}(\vec\mu))$, followed by the tropical branch (resp. source) map, so $br^{trop} = trop_\Sigma\circ br^{\Sigma}$ (resp. $src^{trop} = trop_\Sigma\circ src^{\Sigma}$).
\end{theorem}

Note that the tautological branch morphism is toroidal: it is given analytically locally by a \textit{dominant} equivariant morphism of toric varieties. However, the source map is not toroidal, but locally analytically given by an equivariant morphism of toric varieties, see Definition~\ref{def: loc-an-toric} and~\cite[Remark 5.3.2]{ACP}.

It is proved in~\cite{ACP} that the skeleton $\overline{\Sigma}(\Mbar^{an}_{g,n})$ is identified with $\Mbar^{trop}_{g,n}$. In other words, the analogue of the map $trop_\Sigma$ for $\Mbar_{g,n}$ is an isomorphism. A crucial aspect of the present work is analyzing precisely how the skeleton relates to $\Hbar^{trop}_{g\to h,d}(\vec\mu)$, and indeed the naive extension of the result in loc. cit. does not hold.  The map $trop_\Sigma$ records the combinatorial data of an admissible cover. In particular, the failure of $trop_\Sigma$ to be an isomorphism is due to two phenomena. 
\begin{enumerate}[(A)]
\item Given a weighted dual graph $\Gamma$, there is a unique irreducible stratum of the moduli space $\Mbar_{g,n}$ corresponding to it. In our setting, for one combinatorial type $\Theta = [\Gamma_{src}\to\Gamma_{tgt}]$ for a cover, there are multiple  irreducible components in the stratum  in $\Hbar_{g\to h,d}(\vec\mu)$ having dual graph $\Theta$. In particular, there are multiple zero strata corresponding to the same combinatorial data.
\item The stack of admissible covers arises as the normalization of the stack of generalized Harris--Mumford covers $\HMbar_{g\to h,d}(\vec\mu)$.
There are multiple cones of the skeleton $\Sigma(\Hbar^{an}_{g\to h,d}(\vec\mu))$  (corresponding to the multiple analytic branches at a point of the moduli space $\HMbar_{g\to h,d}(\vec\mu)$) which  all map isomorphically to the same cone of the tropical moduli space. We discuss this in detail in Section~\ref{sec: trop-sigma}.
\end{enumerate}

All the results share two common ingredients. The first is the technology developed in~\cite{ACP,Thu07} that allows to study skeletons of toroidal compactifications of Deligne--Mumford stacks over trivially valued fields. The second is a careful study of the boundary stratification and deformation theory of the stack of admissible covers from~\cite{ACV, HM82, Moch}. 

\subsection{Context and Motivation}
Classical Hurwitz theory studies ramified maps between algebraic curves. Hurwitz numbers count the number of covers of a genus $h$ curve by a genus $g$ curve, with prescribed degree, ramification data, and branch points. As often is the case in enumerative geometry, there is a tight dictionary between the enumerative data of Hurwitz numbers and the intersection theory on the moduli spaces parameterizing Hurwitz covers. 

{\it Hurwitz spaces}, which parameterize covers of smooth curves by smooth curves,  are not proper. For many applications, including enumerative geometry, it is desirable to compactify the Hurwitz space. There are multiple approaches to compactifying this space, each with its pros and cons. In this work we focus on the compactification by {\it admissible covers}.

 The notion of admissible covers was first introduced by Harris and Mumford in~\cite{HM82}. 
The fundamental idea is that source and target curves must degenerate together. Branch points are not allowed to come together; as branch points approach, a new component of the base curve ``bubbles off'', and simultaneously the source curve splits into a nodal curve.

The admissible covers that Harris and Mumford consider are covers of genus $0$ curves, having only simple ramification --- namely, such that the all ramification profiles are given by $(2,1,1,\ldots,1)$. Their work is generalized by Mochizuki in his thesis~\cite{Moch}. Mochizuki uses logarithmic geometry to understand the geometry of the admissible cover space for covers of arbitrary genus and arbitrary ramification profiles. Abramovich, Corti, and Vistoli, in~\cite{ACV}, reinterpret admissible covers using the theory of twisted stable maps to classifying stacks $BS_d$. A map from a curve $C$ to the stack $BS_d$ produces, by definition, a principal $S_d$ bundle on $C$. Given such a principal $S_d$-bundle $P\to C$, one can associate a finite \'etale cover of degree $d$, $D\to C$, where $D = P/S_{d-1}$. In fact, this gives an equivalence of categories between principal $S_d$-bundles and finite \'etale covers of $C$ of degree $d$. By allowing orbifold structure at points  and nodes of $C$ one introduces ramification over such points, and a map from the orbicurve to $BS_d$  corresponds to an admissible cover $D\to C$ where $D = P/S_{d-1}$. The orbifold structure at the nodes of $C$ is required to be balanced,  which is precisely the condition allowing the nodes to be smoothly deformed.
The Abramovich--Corti--Vistoli stack of twisted stable maps  is the normalization of  the Harris--Mumford admissible covers. 

The Hurwitz enumeration problem provides  deep connections between the representation theory of the symmetric group, enumerative geometry, intersection theory on moduli spaces, and combinatorics. For example, see~\cite{GJ1,GJ2,OP06}. 
Tropical Hurwitz theory was first introduced in~\cite{CJM1}, where the case of double Hurwitz numbers for genus $0$ targets is investigated. Further steps in the theory of tropical Hurwitz covers have since been made by Bertrand, Brugall\'e, and Mikhalkin~\cite{BBM}, and by Buchholz and the second author~\cite{BM13}.

At the base of any successful application of tropical methods to enumerative geometry are so-called \textit{correspondence theorems}, which establish equality between classical and tropical enumerative invariants. The first instance of such a result was demonstrated by Mikhalkin~\cite{Mi03}, in his study of the Gromov--Witten invariants of toric surfaces. His correspondence result follows from a direct bijection between the set of algebraic curves satisfying fixed incidence conditions and the (weighted) set of corresponding tropical curves. Subsequent breakthroughs in tropical enumerative geometry have shared this feature of establishing direct set-theoretic bijections between the tropical and classical objects. Enumerative invariants often represent degree zero Chow cycles on a natural moduli space, and traditional correspondence theorems do not link the classical and tropical problems at the level of moduli spaces. Tropical moduli spaces and their intersection theory have been studied in order to express tropical enumerative invariants, analogously to the algebraic setting, as intersection products on a suitable moduli space~\cite{GKM07}. In light of this, it is natural to seek an understanding of the equality of classical and tropical enumerative invariants at the level of moduli spaces.
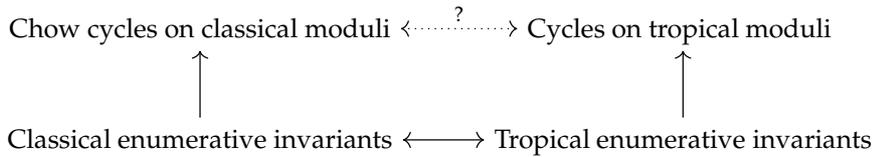
\[
\begin{tikzcd}
\textnormal{\small {Chow cycles on classical moduli}} \ar[leftrightarrow,dotted]{r}{\textnormal ?} & \textnormal{\small {Cycles on tropical moduli }} \\
\textnormal{\small Classical enumerative invariants} \arrow[leftrightarrow]{r} \arrow{u}& \textnormal{\small Tropical enumerative invariants} \arrow{u}
\end{tikzcd}
\]

In this paper, we present the first instance of such a result, by equating classical Hurwitz numbers with tropical ones. We do so by appealing to machinery of Berkovich analytic spaces. The zero cycles representing Hurwitz numbers are obtained as the degree of a naturally defined branch morphism on an appropriate compactification of the Hurwitz space, namely the stack of admissible covers. This recovers previously known results on double Hurwitz numbers~\cite{CJM1}. The present framework also applies equally well to more general settings, such as higher genus targets, and arbitrary ramification profiles. 
In particular, we can also reprove the general correspondence theorem for tropical Hurwitz numbers from \cite{BBM} at the level of moduli spaces.

Our method for comparing classical and tropical moduli spaces of admissible covers follows closely the work of Abramovich, Caporaso and Payne \cite{ACP} on the moduli space of curves.
The moduli space of genus $g$ tropical curves, roughly speaking, parametrizes vertex weighted metric graphs of a given genus~\cite{BMV11, Cap11, CMV12}. This moduli space bears some striking similarities to the moduli space of genus $g$ Deligne--Mumford stable curves. The spaces share the same dimension, and have similar stratifications. These analogies are put on firm ground by realizing the space $\Mbar_{g,n}^{trop}$ as a skeleton of the Berkovich analytification of the stack $\Mbar_{g,n}$. The relationship between $\Mbar_{g,n}$ and $\cM_{g,n}^{trop}$ is not unlike the relationship between a toric variety and its fan. This intuition is made precise by using a natural toroidal structure of $\Mbar_{g,n}$ to obtain a skeleton of $\Mbar^{an}_{g,n}$. Abramovich, Caporaso, and Payne developed techniques, building on work of Thuillier~\cite{Thu07}, to construct the skeleton of the analytification of any toroidal Deligne--Mumford stack, and relate it to the cone complex naturally  associated to the toroidal embedding~\cite{KKMSD}. As stacks of admissible covers are smooth with a toroidal structure induced by the  normal crossing boundary,  these techniques apply directly to our setting.

The study of Hurwitz numbers of $\PP^1$ has also interacted fruitfully with the theory of stable maps and Gromov--Witten invariants of $\PP^1$. In fact, the moduli space of Hurwitz covers sits inside the moduli space of degree $d$ relative stable maps to $\PP^1$, relative to the (special) branch divisor. Applying techniques of Gromov--Witten theory such as Atyiah-Bott localization connected Hurwitz numbers to the tautological intersection theory on the moduli space of curves: the ELSV formula~\cite{ELSV, GV05}, which gives a precise equality between simple Hurwitz numbers and one part Hodge integrals, has been a key ingredient in Okounkov and Pandharipande's proof ~\cite{OP} of Witten's Conjecture/Kontsevich's Theorem.

In~\cite{CJM1}, the first and second authors, together with Paul Johnson, constructed and studied a tropical analogue of the Gromov--Witten moduli stack of stable maps to $\PP^1$, relative to a two-point special branch divisor. The corresponding classical moduli space is a singular, non-equidimensional stack, and does not afford a direct application of the techniques of Abramovich, Caporaso, and Payne.

The space of relative stable maps can be seen as a ``hybrid'' theory, between admissible covers and (absolute) stable maps. More precisely, admissible covers is a theory of relative stable maps when the entire branch divisor on the target curves is made relative (and it is allowed to ``move"). The admissible cover compactification of the Hurwitz scheme admits a rational map to the relative stable map space which is dominant on the main (expected dimensional) component. As a result, we see the study of admissible covers as a natural first step towards a functorial and geometric understanding of tropical relative Gromov--Witten theory. 


\subsection*{Acknowledgements} \noindent
The first author acknowledges with gratitude the support by NSF grant DMS-1101549, NSF RTG grant 1159964, and the second author the support by DFG grant MA 4797/6-1. The third author acknowledges many fruitful conversations with Tyler Foster, Dave Jensen, Douglas Ortego, Yoav Len, and Martin Ulirsch. The authors thank Sam Payne for introducing them, as well as for several insightful comments. We also thank an anonymous referee for helpful comments on an earlier version.

\section{Combinatorial constructions}\label{sec: comb-const}

A unifying feature of the numerous instances of tropicalization is that they associate combinatorial and polyhedral structures to algebraic varieties. Examples of these are finite and metric graphs, cones, fans, and polyhedral complexes. Often, these structures are finite approximations to an appropriate Berkovich analytic space. In this section we briefly recall the concepts relevant to this work.

\subsection{Dual graphs}

If $C$ is a nodal curve, one may associate a \textit{vertex weighted dual graph} or simply \textit{dual graph} $\Gamma_C$ as follows:
\begin{enumerate}[(i)]
\item the vertices of $\Gamma_C$ are  the irreducible components of $C$;
\item the edges of $\Gamma_C$ are the nodes of $C$; an edge $e$ is incident to a vertex $v$ if the node associated to $e$ is contained in the component corresponding to $v$;
\item a vertex $v$ is given a weight $g(v)$, equal to the geometric genus of the component corresponding to $v$. 
\end{enumerate}

If $C$ is a pointed curve, with marked points $p_1,\ldots, p_n$, then we add an \textit{infinite edge} for each $p_i$, incident to the vertex whose component contains $p_i$. We say that a dual graph is \textit{totally degenerate} if all vertices carry genus $0$. 

We call a dual graph {\it stable} if a nodal curve having that dual graph is stable. In other words, all genus $0$ vertices must be at least trivalent (counting infinite edges), and all genus $1$ vertices must be incident to an edge. Notice that loops and multiple edges are allowed, and a loop contributes two to the valence of its adjacent vertex.

\subsection{Tropical curves and morphisms}
For our purposes, an \textit{$n$-pointed tropical curve} is essentially a metrization of the dual graph of an $n$-pointed nodal curve, where edges corresponding to marked points are metrized as $\RR_{> 0}\cup\{\infty\}$, and the metric must be singular at the ends of these infinite edges.  We sometimes refer to edges corresponding to nodes as \textit{interior edges} to distinguish them from infinite edges. 

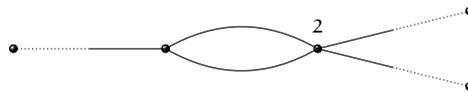
\begin{figure}[h!]
\begin{tikzpicture}

\draw [ball color=black] (-2,0) circle (0.5mm);
\draw [ball color=black] (0,0) circle (0.5mm);
\draw [ball color=black] (2,0) circle (0.5mm);
\draw [ball color=black] (4,0.5) circle (0.5mm);
\draw [ball color=black] (4,-0.5) circle (0.5mm);

\path (0,0) edge [bend left] (2,0);
\path (0,0) edge [bend right] (2,0);

\draw (0,0)--(-1,0);
\draw [densely dotted] (-2,0)--(-1,0);
\draw (2,0)--(3,0.25);
\draw [densely dotted] (4,0.5)--(3,0.25);
\draw (2,0)--(3,-0.25);
\draw [densely dotted] (4,-0.5)--(3,-0.25);

\draw node at (2,0.3) {\tiny$2$};

\end{tikzpicture}
\caption{A tropical $3$-pointed curve of genus $3$, with $2$ interior edges. Unmarked vertex weights are $0$. Lengths of interior edges are $1$.}
\end{figure}

In other words, a finite metric graph $\Gamma$ is a compact 1-dimensional topological space, locally homeomorphic to $S_r$, the ``star with $r$ branches''. Furthermore, there are only finitely many points of $\Gamma$ at which $r\neq 2$. This unique integer $r$ at each point is called the valence of the point. The set of tangent directions at $p\in \Gamma$ is defined as
\[
T_p\Gamma := \varinjlim_{U_0} \pi_0(U_p\setminus p),
\]
where the limit is taken over neighborhoods of $p$ homeomorphic to a star with $r$ branches. A tangent direction may be thought of as a germ of an edge. The size of this set equals the valence of $p$. 

\noindent
A tropical curve $\Gamma$ is a connected metric graph with a weighting $g:\Gamma\to \mathbb Z_{\geq0}$, which is zero outside finitely many points of $\Gamma$. This weight $g(p)$ should be thought of as the genus of a virtual algebraic curve lying above $p$. (This intuition can be made concrete in terms of metrized complexes of curves and Berkovich skeleta, in the sense of~\cite{ABBR}.) The genus of a graph $\Gamma$ is given by 
\[
g(\Gamma) = h^1(\Gamma)+\sum_{p\in \Gamma} g(p). 
\]
Throughout the text, we consider \textit{$n$-pointed tropical curves}, i.e.\ tropical curves with $n$ marked infinite edges.
An ($n$-pointed) tropical curve is \textit{stable} if every genus $0$ vertex is at least trivalent.

One can associate to any (pointed) tropical curve $\Gamma$ a finite graph, which we refer to as a \textit{combinatorial type}, by taking its minimal finite graph model. That is, we take the vertices to be those points of valence different from $2$, or whose genus is nonzero. The edges are formed in the obvious way. Conversely, given a combinatorial type, a metrization is an assignment of lengths to the edges. Infinite edges are metrized as $[0,\infty]$.

\subsubsection{Harmonic morphisms of tropical curves} Let $\varphi:\Gamma'\to \Gamma$ be a morphism of metric graphs, such that the restriction of $\varphi$ to an edge is affine linear. For an edge $e'$ of $\Gamma'$,  the \textit{expansion factor along $e'$} of $\varphi$ is defined to be the ratio of lengths $l(\varphi(e'))/l(e')$ and it is required to be an integer.  We denote this number $d_{e'}(\varphi)$ and also  call it the slope of $\varphi$ along $e'$. We also fix slopes $d_{e'}(\varphi)$ for along each infinite edge $e'$.

A morphism of tropical curves is a harmonic morphism between the underlying metric graphs with integer affine slopes, in the sense of~\cite[Section~1]{ABBR}. We recall that a 
morphism $\varphi$ as above is \textit{harmonic at $p'$} if for each  tangent direction $e\in T_{\varphi(p')}\Gamma$, the number
\begin{equation}\label{eq:deg}
d_{p'}(\varphi):=\sum_{{\substack{{e'\in T_{p'}\Gamma'} \\ {e'\mapsto e}}}} d_{e'}(\varphi),
\end{equation}
is independent of $e$. In other words, the sum of outgoing expansion factors at $p'$ along tangent directions mapping to a chosen tangent direction $e$ is independent of $e$. The morphism $\varphi$ is harmonic if it is surjective and harmonic at all points of $\Gamma'$ (i.e.\ non-constant). 

The integer $d_{p'}(\varphi)$ is called the  {\it local degree} of $\varphi$ at
$p'$.  Note that a harmonic morphism has itself a degree, defined to be, for any vertex $p$ in the base graph, the sum of the local degrees of $\varphi$ for all vertices $p'$ mapping to $p$.

\begin{example}
Consider the non-harmonic morphism depicted in Figure~\ref{fig: non-harmonic}. Observe that on one side of the central vertex in the target, the sum of degrees mapping to it is $5$, while on the other side, it is $4$. In particular, there is no well defined notion of degree for such a map.

\begin{figure}[h!]
\begin{tikzpicture}
\draw [ball color = black] (0,0) circle (0.4mm);
\draw [ball color = black] (135:1) circle (0.4mm);
\draw [ball color = black] (45:1) circle (0.4mm);
\draw [ball color = black] (270:1) circle (0.4mm);
\draw (45:1)--(0,0);
\draw (135:1)--(0,0);
\draw (270:1)--(0,0);
\draw [->] (1.5,0)--(2,0);

\draw node at (160:0.5) {\tiny $3$};
\draw node at (20:0.5) {\tiny $2$};
\draw node at (250:0.5) {\tiny $4$};

\draw [ball color = black] (3,0) circle (0.4mm);
\draw [ball color = black] (3,0.707) circle (0.4mm);
\draw [ball color = black] (3,-1) circle (0.4mm);
\draw (3,0)--(3,0.707);
\draw (3,0)--(3,-1);

 \begin{scope}[shift={(8,0)}]
\draw [ball color = black] (0,0) circle (0.4mm);
\draw [ball color = black] (135:1) circle (0.4mm);
\draw [ball color = black] (45:1) circle (0.4mm);
\draw [ball color = black] (270:1) circle (0.4mm);
\draw (45:1)--(0,0);
\draw (135:1)--(0,0);
\draw (270:1)--(0,0);
\draw [->] (1.5,0)--(2,0);

\draw node at (160:0.5) {\tiny $3$};
\draw node at (20:0.5) {\tiny $2$};
\draw node at (250:0.5) {\tiny $5$};

\draw [ball color = black] (3,0) circle (0.4mm);
\draw [ball color = black] (3,0.707) circle (0.4mm);
\draw [ball color = black] (3,-1) circle (0.4mm);
\draw (3,0)--(3,0.707);
\draw (3,0)--(3,-1);
  \end{scope}

\end{tikzpicture}
\caption{Depicted on the left is a non-harmonic morphism of graphs, and on the right a harmonic morphism. The numbers on the source graph indicate the expansion factors along the corresponding edge. The morphism on the left is not harmonic in a neighborhood of the central vertex.}
\label{fig: non-harmonic}
\end{figure}
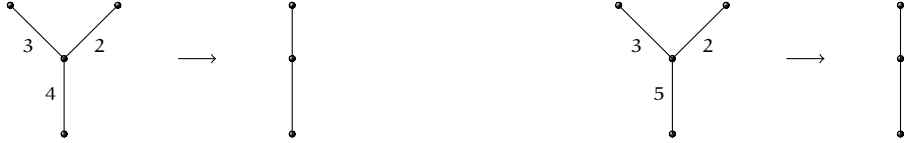

\end{example}

\subsection{Cones and cone complexes}
A \textit{polyhedral cone with integral structure} is a pair $(\sigma, M)$ consisting of a topological space $\sigma$, together with a finitely generated abelian group $M$ of continuous real valued functions on $\sigma$, such that the natural map $\sigma\to Hom(M,\RR)$ is a homeomorphism onto a (strictly convex) polyhedral cone. We only consider \textit{rational} cones, i.e. cones whose image is rational with respect to the dual lattice $Hom(M,\ZZ)$. A good example to have in mind is the cone defining an affine toric variety, where $M$ is the character lattice of the dense torus. 

Let $\sigma$ be a cone. The \textit{dual monoid} $S_\sigma$ of $\sigma$ consists of those functions $f\in M$ that are nonnegative on $\sigma$. We can recover the original cone from its dual monoid as
\[
\sigma = Hom(S_\sigma, \RR_{\geq 0}),
\]
where $\RR_{\geq 0}$ is taken with its usual additive monoid structure. The associated \textit{extended cone} is 
\[
\overline{\sigma} = Hom(S_\sigma, \RR_{\geq 0}\cup\{\infty\}).
\]

Every cone (resp. extended cone) inherits a topology by realizing it as a subspace of the space $\RR_{\geq 0}^{S_\sigma}$ (resp. the space $(\RR_{\geq 0}\sqcup \{\infty\})^{S_\sigma}$).

A \textit{(rational polyhedral) cone complex} is a topological space obtained from a finite disjoint union of polyhedral cones with integral structures, by gluing cones along isomorphic faces.
Fans are of course examples of cone complexes. We remark however that, in a cone complex, the intersection of  cones is allowed to be a union of faces, rather than a single face of each. Additionally, a cone complex makes no reference to an abstract vector space in which it is embedded. See~\cite{KKMSD} for further details. \textit{Extended cone complexes} are obtained analogously from extended cones.

A useful tool for us is the process of barycentric subdivision. 
The barycenter of a cone $\sigma$ is  the ray in its interior spanned by the sum of the primitive generators of the one-dimensional faces of $\sigma$. The \textit{barycentric subdivision} of a cone complex $\Sigma$ is the iterated stellar subdivision of cones in $\Sigma$, in decreasing order of dimension. A more elegant, if less concrete, definition of the barycentric subdivision is that it is obtained as the poset of chains in the face poset of $\Sigma$, ordered by inclusion. The following proposition is often useful.

\begin{proposition}[\cite{AMR}]
Let $\Sigma$ be any cone complex. Then the barycentric subdivision of $\Sigma$ is isomorphic to a simplicial fan. 
\end{proposition}

A \textit{face morphism} is a morphism between cone complexes such that every cone maps isomorphically onto a cone in the image. 

\begin{definition}
A \textit{generalized cone complex} is an arbitrary finite colimit of cones $\sigma_i$ and face morphisms $\psi_i$,
\[
\Sigma = \varinjlim (\sigma_i,\psi_i).
\]
\end{definition}

These more general objects are built to allow these additional operations: gluing  two isomorphic faces of a single cone, and taking quotients of cones by a group of automorphisms. 

Here, generalized cone complexes appear as Berkovich skeleta of moduli spaces of admissible covers, resp.\ as the corresponding tropical moduli spaces.

\section{Classical and tropical admissible covers}\label{sec: bkgrnd}

\subsection{Classical admissible covers and their moduli}

Let $(C,p_1,\ldots,p_n)$ be a genus $g$ $n$-pointed stable nodal curve.
\begin{definition}
An \textit{admissible cover} $\pi: D\to C$ of degree $d$ is a finite morphism of pointed curves such that:
\begin{enumerate}[(i)]
\item The map $\pi$ restricted to the complement of the inverse image of the marked points and nodes is \'etale of constant degree $d$. 
\item All inverse images of marked points of $C$ are marked in $D$.
\item The set of nodes of $D$ is precisely the preimage under $\pi$ of the set of nodes of $C$.

\item Over a node, \'etale locally, $D$, $C$ and $\pi$ are described by
\begin{eqnarray*}
D&:& y_1y_2 = a\\
C&:& x_1x_2 = a^\ell\\
\pi&:& x_1 = y_1^\ell, \ \ x_2 = y_2^\ell.
\end{eqnarray*}
for some positive integer $\ell\leq d$.
\end{enumerate}
\end{definition}

\begin{remark}
Intuitively, condition \textit{(i)} means the branch locus of $\pi$ is contained in the union of marked points and nodes of $C$. Condition   \textit{(ii)} spells that we want to distinguish the inverse images of each branch point: given a branch point, the markings of its inverse images tell apart points with the same ramification index.
Condition \textit{(iii)} and \textit{(iv)}  amount to saying that nodal covers arise as limit of covers of smooth curves where source and target degenerate simultaneously.  In particular  we have a natural kissing condition: if $\pi$ is lifted to $\tilde{\pi}$ on the normalizations of D and C, then for each node of D, the
ramification indices of $\tilde{\pi}$ at the two points lying above the node coincide.
\end{remark}

There are several slight variations of moduli spaces of admissible covers. For instance, the base curve can be fixed, or allowed to vary in moduli. Fixing both the base curve and the branch locus, the moduli space becomes zero dimensional; the degree of the fundamental class is called a Hurwitz number. 


Fix a vector of partitions $\vec\mu = (\mu^1,\ldots,\mu^r)$ of an integer $d>0$.

\begin{definition}
Fix $(r+s)$ points $p_1,\ldots, p_r,q_1,\ldots,q_s$ on a smooth genus $h$ curve $C$. The Hurwitz number, denoted $h_{g\to h,d}(\vec\mu)$, is the weighted number of degree $d$ covers $[\pi:D\to C]$ such that $\pi$ is unramified over the complement of $\{p_i,q_j\}_{i,j}$, with ramification profile $\mu_i$ over $p_i$ and simple ramification over $q_j$. Each cover is weighted by $1/|Aut(\pi)|$. 
\end{definition}

Notice that we mark all the preimages of branch points, this convention simplifies combinatorial aspects in the tropical version.

\begin{example}
Here are some examples of Hurwitz numbers.
\begin{eqnarray*}
h_{0\to 0,d}((d),(d)) &=& \frac{1}{d}\\
h_{1\to 0,2} &=& \frac{1}{2}\ \  \textnormal{(all ramification is simple)}\\
h_{1\to 0,4}((2,2),(4))&=&14\cdot 2\cdot 2^3.
\end{eqnarray*}
The first  two numbers  are easily computed by counting monodromy representations (\cite{renzosbook}). For an example computation of the third number via monodromy graphs, see~\cite[Example 4.5]{CJM1}. We have a factor of $2$ because we mark the two points giving the profile $(2,2)$, contrary to~\cite{CJM1}, and a factor of $2^3$ because the unramified inverse images above each branch point are also marked. 
\end{example}



The space of admissible covers is in general a non-normal stack, however the normalization is always smooth. A modular interpretation of the normalization as the stack of twisted stable maps to the classifying stack $BS_d$ was given by Abramovich, Corti and Vistoli in~\cite{ACV}.  Considering partitions of $d$ as labelling connected components of the inertia stack of $BS_d$,  the space $ACV_{\substack{g\to h}, {d}}(\vec \mu)$ is the component of the stack
$\overline{\mathcal{M}}_{h,r+s}(BS_d,0)$ identified by the inertial conditions $\prod_{i=1}^r ev_i^\ast(\mu_i)\prod_{i=1}^s ev_{r+i}^\ast((2,1^{d-2}))$. The number $s$ of $\mathbb{Z}/2$ twisted points pulling back the class of a transposition is related to $g$ via the Riemann-Hurwitz formula.
We  provided a precise statement for the benefit of  readers who are already familiar with this language: however we will not make use of this language in any sophisticated way.

\begin{definition}
We denote by $\overline{\mathcal H}_{\substack{g\to h}, {d}}(\vec \mu)$ the cover of $ACV_{\substack{g\to h}, {d}}(\vec \mu)$ obtained by marking the inverse images of all marked points in the corresponding admissible covers. The coarse objects parameterized  are admissible covers of degree $d$ of a genus $h$ curve with $(r+s)$ marked branch points $(p_1,\ldots,p_r,q_1,\ldots,q_s)$, by curves of genus $g$, having ramification profiles $\mu^i$ over $p_i$ and simple ramification over $q_i$ (and no further ramification). Denote by ${\mathcal H}_{g\to h,d}(\vec\mu)$ the open substack parametrizing covers whose source and target are smooth curves. 
\end{definition}

\begin{remark}
We  refer to the normalized stack as the stack of admissible covers.  When we need to specifically point our attention to the Harris--Mumford spaces of admissible covers,  we  explicitly say so and denote this stack by $\HMbar_{g\to h,d}(\vec\mu)$. The open parts of these moduli spaces coincide, that is, $\cH\cM_{g\to h,d}(\vec\mu) \cong \cH_{g\to h,d}(\vec\mu)$. In other words, the non-normality manifests in how boundary strata intersect. We will witness this non-normality (as well as why the normalization is smooth) when we study the local rings of these moduli spaces in Section~\ref{sec: def2}. 

\end{remark}

\noindent
\begin{convention}
The number $s$ always denotes the number of simple branch points, the number $h$ the genus of the base curve. To avoid burdensome notation, we suppress $h$ and $s$  and use  $\overline{\mathcal H}_{{g}, {d}}(\vec \mu)$ to denote our moduli space, with the understanding that $h$ may be arbitrarily chosen but is fixed, and $s$ is determined by the Riemann--Hurwitz formula.  
\end{convention}

\subsubsection{Toroidal embeddings} A toroidal scheme is a pair $U\hookrightarrow X$ which ``locally analytically'' looks like the inclusion of the dense torus into a toric variety. That is, at every point $p\in X$, there is an \'etale (or formal) neighborhood $\varphi:V\to X$, which admits an \'etale map $\psi: V\to V_\sigma$ to an affine toric variety, such that 
\[
\psi^{-1}T = \varphi^{-1}U,
\]
where $T$ is the dense open torus.  

Let $\mathscr X$ be a Deligne--Mumford stack over a field, with coarse space $X$. Let $\mathscr U\subset \mathscr X$ be an open substack. For any morphism from a scheme $h: V\to \mathscr X$, denote $\mathscr U_V\subset V$ the pre-image of $\mathscr U$ in $V$.

\begin{definition}
The inclusion $\mathscr U\hookrightarrow \mathscr X$ is a \textit{toroidal embedding} of Deligne--Mumford stacks if, for every morphism from a scheme $V\to \mathscr X$, the inclusion $\mathscr U_V\hookrightarrow V$ is a toroidal embedding of schemes. 
\end{definition}

Toric varieties are obvious examples of toroidal embeddings. Another  relevant example is the inclusion of the complement of a normal crossings divisor, $(X-D)\hookrightarrow X$. In fact, in this case, all local toric models can be taken to be affine spaces. The moduli space of stable pointed curves, $\Mbar_{g,n}$ is hence an  example of a toroidal Deligne--Mumford stack, since the boundary $\Mbar_{g,n}\setminus \cM_{g,n}$ is a divisor with normal crossings. If in addition $g = 0$, the boundary divisor has strict normal crossings (i.e. the irreducible components have no self intersections). Similarly, the stack $\overline{\mathcal H}_{g,d}(\vec \mu)$ is a smooth Deligne--Mumford stack, and the boundary $\overline{\mathcal H}_{g,d}(\vec\mu)\setminus{\mathcal H}_{g,d}(\vec\mu)$ is a divisor with normal crossings, allowing us to apply the techniques developed by Abramovich, Caporaso, and Payne in this setting. For admissible covers as well, if $h = 0$, the boundary divisor has strict normal crossings. 

\subsection{Tropical admissible covers and their moduli}
\subsubsection{Dual graphs of covers}\label{sec: dual-graph-covers} Just as one can associate a dual graph to any nodal pointed curve, given an admissible cover, we can associate to it the dual graphs of source and target, and a map between them. This map is a well defined morphism of graphs by the axioms placed on admissible covers. 
Irreducible components map to irreducible components, marked points map to marked points, and nodes map to nodes, inducing maps on vertices, infinite edges, and edges, respectively. Note that no edges are contracted in the map of graphs, since nodes map to nodes in an admissible cover.

\begin{enumerate}
\item[] \textbf{Source and target curves.} Take the dual graph of the source and target curves in the above sense. Call these graphs $\Gamma_{src}$ and $\Gamma_{tgt}$ respectively. Recall here that all branch and ramification points are marked. 
\item[] \textbf{The map.} For an admissible cover, a component of the source maps onto precisely one component of the target, yielding a map of vertices. Since nodes map to nodes, edges map to edges.
\item[] \textbf{Ramification.} We mark edges of $\Gamma_{src}$ with integers recording the ramification at the corresponding  
node or marked point of the source curve. That is, if an edge $\widetilde e$ of $\Gamma_{src}$ maps to an edge $e$ of $\Gamma_{tgt}$, this corresponds to a special point $\widetilde p$ of the source curve mapping to a special point $p$ of the target. We decorate $\widetilde e$ with the ramification index of the map at $\widetilde p$.
\end{enumerate}

%

\subsubsection{Tropical admissible covers} \label{subsec-tropadm}
We now recall the notion of a tropical admissible cover, slightly adapting Caporaso's definition in~\cite[Section 2]{cap:gonality}. We wish to study covers of genus $g$ tropical curves, with prescribed ramification data over $r$ points and simple ramification over the remaining $s$ points. 
We say that a  map of tropical curves satisfies the \textit{local Riemann--Hurwitz condition} if, when $v'\mapsto v$ with local degree $d$, then
\[
2-2g(v') = d(2-2g(v)) - \sum (m_{e'}-1),
\]
where $e'$ ranges over edges incident to $v'$, and $m_{e'}$ is the expansion factor of the morphism along $e'$. 
\begin{definition}
A \textit{Hurwitz cover} of a tropical curve is a harmonic map of tropical curves that satisfies the local Riemann--Hurwitz equation at every point.
\end{definition}
\noindent

\subsubsection{Constructing the tropical moduli space: fixed combinatorial type} 

Throughout this subsection, we fix a \textit{degree} $d>0$ and vector of partitions of $d$ denoted $\vec \mu = (\mu^1,\ldots,\mu^r)$. Let $\ell(\mu^i)$ denote the number of parts of the $i^{\mathrm{th}}$ partition. Furthermore fix two integers $g$ and $h$ which are to be the genera of the source and target curve respectively. The Riemann--Hurwitz formula determines a number of \textit{simple branch points} for a curve with ramification profile $\vec\mu$, and we let $s$ be this number. Finally, fix $n = \sum_i \ell(\mu^i)+s(d-1)$.

We will now construct a tropical Hurwitz space $\cH^{trop}_{g,d}(\vec\mu)$ of degree $d$ tropical covers of genus $h$ $(r+s)$-marked curves by genus $g$, $n$-marked curves, having ramification profiles over the $i^{\mathrm{th}}$ marked point given by the partition $\mu^i$. This construction is a variation on the procedure used in~\cite{ACP}, which we briefly outline.

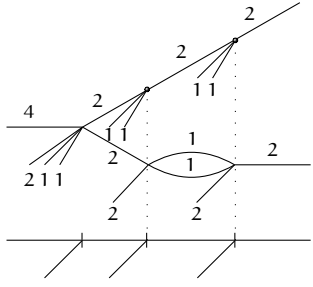
\begin{figure}[bt]
\begin{tikzpicture}

\draw (-1,0)--(0,0);
\draw (0,0)--(30:3.3);
\draw (0,0)--(-30:1);
\path (-30:1) edge [bend left] (2,-0.5);
\path (-30:1) edge [bend right] (2,-0.5);
\draw (2,-0.5)--(3,-0.5);

\draw (0,0)--(-0.5,-0.5);
\draw (0,0)--(-0.7,-0.5);
\draw (0,0)--(-0.3,-0.5);
\draw (0.85,0.5)--(0.35,0);
\draw (-30:1)--(0.4,-1);
\draw (2.01,1.15)--(1.51,0.65);
\draw (2,-0.5)--(1.5,-1);
\draw (0.85,0.5)--(0.55,0);
\draw (2.01,1.15)--(1.71,0.65);


\draw node at (-0.7,0.2) {\tiny $4$};
\draw node at (1.3,1.0) {\tiny $2$};
\draw node at (.4,-0.4) {\tiny $2$};
\draw node at (1.42,-0.5) {\tiny $1$};
\draw node at (1.45,-0.15) {\tiny $1$};
\draw node at (2.5,-0.3) {\tiny $2$};

\draw node at (-0.7,-0.7) {\tiny $2$};
\draw node at (-0.3,-0.7) {\tiny $1$};
\draw node at (-0.5,-0.7) {\tiny $1$};
\draw node at (2.2,1.5) {\tiny $2$};
\draw node at (0.2,0.35) {\tiny $2$};
\draw node at (.55,-0.1) {\tiny $1$};
\draw node at (.35,-0.1) {\tiny $1$};
\draw node at (1.71,0.5) {\tiny $1$};
\draw node at (1.51,0.5) {\tiny $1$};
\draw node at (1.5,-1.15) {\tiny $2$};
\draw node at (0.4,-1.15) {\tiny $2$};

\draw [-|]  (-1,-1.5)--(0,-1.5);
\draw [-|]  (0,-1.5)--(0.85,-1.5);
\draw [-|]  (0.85,-1.5)--(2.01,-1.5);
\draw [-]  (2.01,-1.5)--(3,-1.5);

\draw [loosely dotted] (0.85,-1.5)--(0.85,0.5);
\draw [ball color=black] (0.85,0.5) circle (0.3mm);
\draw [loosely dotted] (2.01,-1.5)--(2.01,1);
\draw [ball color=black] (2.01,1.15) circle (0.3mm);

\draw(0,-1.5)--(-0.5,-2);
\draw (0.85,-1.5)--(0.35,-2);
\draw (2.01,-1.5)--(1.51,-2);
\end{tikzpicture}
\caption{A combinatorial type in the space $\cH^{trop}_{1\to 0,4}((4),(2,2))$ representing a top dimensional stratum. Vertices are undecorated to mean that their genus is $0$. The markings of the infinite edges are not depicted to avoid an overburdened figure.
Edges of the top graph are decorated with their expansion factor. Note that the local degree of the map at the leftmost vertex is $4$, whereas at all other vertices is $2$. }
\label{fig:combtype}
\end{figure}

A \textit{combinatorial type} $\Theta$ of  a tropical admissible cover is the data of a tropical admissible cover without the metric, as illustrated in Figure \ref{fig:combtype}. That is, $\Theta$ consists of a morphism of finite graphs, together  with a decoration of the edges of the source with integers recording the expansion factors. Such decoration makes the morphism harmonic. If $\Gamma_{src}\to \Gamma_{tgt}$ is a tropical admissible cover, we denote the associated combinatorial type by $[\Gamma_{src}\to \Gamma_{tgt}]$. If $e'$ is an edge of the finite graph $\Gamma_{src}$, we denote by $d_{e'}$ the expansion factor on the metrized edge corresponding to $e'$, on any tropical admissible cover of type $\Theta$. When we wish to speak of a combinatorial type without reference to the tropical cover that it came from, we refer to the type as a \textit{combinatorial admissible cover}. A metrization of the base graph fully determines the tropical Hurwitz cover, as we see in the next lemma. 

\begin{lemma}\label{lem:det}
Let $\Theta = [\Gamma_{src}\to \Gamma_{tgt}]$ be a combinatorial type for a tropical admissible cover. Given a metrization $\ell:E(\Gamma_{tgt})\to \RR_{\geq 0}\sqcup \{\infty\}$, there exists a unique metrization of $\Gamma_{src}$ making the resulting map an admissible cover of combinatorial type $\Theta$.
\end{lemma}

\begin{proof}
Let $e'$ be an edge of $\Gamma_{src}$ mapping to an edge $e$ of $\Gamma_{tgt}$. The type $\Theta$ carries the datum of an expansion factor $d_{e'}$ on $e$. If $\ell(e)$ is the length of the edge $e$, then $\ell(e') = \ell(e)/d_{e'}$ is the unique length on $e'$ that makes the resulting map piecewise linear of expansion factor $d_{e'}$. The maps on each edges glue, and satisfy the harmonicity condition by virtue of $\Theta$ being a combinatorial type. 
\end{proof}

%
%
%
%
%

\begin{definition}
Let $\Theta = [\theta:\Gamma_{src}\to\Gamma_{tgt}]$ be a combinatorial type for an admissible cover. Then an automorphism of $\Theta$ is the data of a commuting square of automorphisms of base and target
\[
\begin{tikzcd}
\Gamma_{src} \arrow{r}{\varphi_{src}} \arrow{d}[swap]{\theta} & \Gamma_{src} \arrow{d}{\theta}\\
\Gamma_{tgt} \arrow{r}[swap]{\varphi_{tgt}} & \Gamma_{tgt},
\end{tikzcd}
\]
where $\varphi_{src}$ and $\varphi_{tgt}$ are automorphisms, and $\varphi_{src}$ is required to preserve expansion factors on edges. That is, for any edge $e'\in \Gamma_{src}$, we have $d_{e'} = d_{\varphi(e')}$. 
\end{definition}

We use $Aut(\Theta)$ to denote the (finite)  group of automorphisms of $\Theta$. We denote by $Aut_0(\Theta)$ the subgroup of $Aut(\Theta)$ where $\varphi_{tgt}$ is the identity, that is, automorphisms of $\Gamma_{src}$ which cover the identity map.


\subsubsection{Hurwitz existence and local Hurwitz numbers} Let $\Theta = [\theta: \Gamma_{src}\to \Gamma_{tgt}]$ be a combinatorial type for an admissible cover. Associated to a vertex $\widetilde v$ of $\Gamma_{src}$, are \textit{local Hurwitz numbers}. That is, for every $\widetilde v\in \theta^{-1}(v)$, there is a Hurwitz number $h_{g\to h,d}(\vec \mu)$. Here, $g$ is the genus of $\widetilde v$, $h$ is the genus of $v$, $d$ is the local degree of $\theta$ at $v$, and $\vec \mu$ is given by the expansion factors along tangent directions at $\widetilde v$. We pause to note that such a local Hurwitz number can be zero. Finding a characterization of when discrete invariants compatible with the Riemann-Hurwitz equation give rise to a non-zero Hurwitz number is an open problem referred to as the {\it Hurwitz existence} problem.
It is sometimes  convenient to study all local Hurwitz numbers  for vertices above a given vertex simultaneously. We denote by $H(v)$ the product of all local Hurwitz numbers over vertices lying above $v$. 

Let $B$ be the number of interior edges of the 
target curve of $\Theta$. By Lemma \ref{lem:det}, a cover in this combinatorial type is determined by a choice of length in $\RR_{\geq 0}$ for each edge. Denote by $\sigma_\Theta = (\RR_{\geq0})^B$. The moduli space $\cH^{trop}_\Theta$ parametrizing covers with combinatorial type $\Theta$ is defined to be $ \sigma_\Theta/Aut(\Theta)$. Denote by $\overline \sigma_\Theta$ the canonical compactification $(\RR_{\geq0}\cup {+\infty})^B$. The extended cone $\Hbar^{trop}_\Theta$ is the quotient of  $\overline \sigma_\Theta$ by $Aut(\Theta)$. We only consider cones such that $H(v)$ is nonzero for all vertices $v$ in $\Gamma_{tgt}$. 

\subsubsection{Constructing the tropical moduli space: graph contractions and gluing}\label{sec: contractions}

A weighted graph contraction of $\Gamma$ is a composition of edge contractions of the underlying graph $\alpha: \Gamma\to \widehat\Gamma$, endowed with a canonical genus function, $g_{\widehat\Gamma}(v) = g_\Gamma(\alpha^{-1}v)$. 

\begin{proposition}
Given a combinatorial admissible cover $\theta: G_{src}\to G_{tgt}$, a graph contraction $G_{tgt}\to \widehat G_{tgt}$ induces a graph contraction $\widehat G_{src}$ of $G_{src}$, together with a combinatorial admissible cover $\widehat G_{src}\to \widehat G_{tgt}$. 
\end{proposition}

\begin{proof}
It suffices to prove the proposition for a single edge contraction $G_{tgt}\to \widehat G_{tgt}$, contracting $e$. Define the contraction $G_{src}\to\widehat G_{src}$ by contracting every edge $e'$ that maps to $e$. The graph $\widehat G_{src}$ inherits a canonical genus function as above. Define $\widehat \theta: \widehat G_{src}\to \widehat G_{tgt}$ as follows. For a vertex $\widehat v$ of $\widehat G_{src}$, let $v$ be any lift of $v$ under the map $V(G_{src})\to V(\widehat G_{src})$. Define $\widehat \theta(\widehat v)$ to be the image of $v$ under the composite
\[
G_{src}\to G_{tgt}\to \widehat G_{tgt}.
\]
It is clear that this map is well defined. An edge of $\widehat G_{src}$ corresponds to a unique edge of $G_{src}$ and thus determines a map of graphs $\widehat \theta: \widehat G_{src}\to \widehat G_{tgt}$. The expansion factors on $\widehat G_{src}$ coincide with the expansion factors of the corresponding edges on $G_{src}$. The result follows.
\end{proof}

In other words, contractions of covers are fully determined by contractions of their 
target graphs.

Let $\Theta = [\theta:\Gamma_{src}\to \Gamma_{tgt}]$ be a combinatorial type of tropical admissible covers, such that $H(v)$ is nonzero for all $v\in \Gamma_{tgt}$. Every automorphism of the combinatorial type $\Theta$ determines an automorphism of the associated cone $\sigma_\Theta$. Moreover, every graph contraction of combinatorial types $\Theta\to \Theta'$ determines a map of cones $\sigma_{\Theta'}\to \sigma_\Theta$. It is clear that ranging over $\Theta$ as above, this forms a directed system of topological cones with integral structure.

The moduli space of tropical admissible covers is constructed as the topological colimit 
\[
\cH^{trop}_{g,d}(\vec \mu) = \varinjlim (\sigma_\Theta,j_\omega),
\] 
where $\Theta$ ranges over the combinatorial types, and $j_\omega$ is a contraction of combinatorial types or an automorphism of types. \textit{After performing a barycentric subdivision,} the space $\cH^{trop}_{g,d}(\vec\mu)$ inherits the structure of a cone complex with integral structure. The extended cones $\overline\sigma_\Theta$ are glued similarly to obtain $\Hbar^{trop}_{g,d}(\vec \mu)$, which naturally carries an extended cone complex structure.

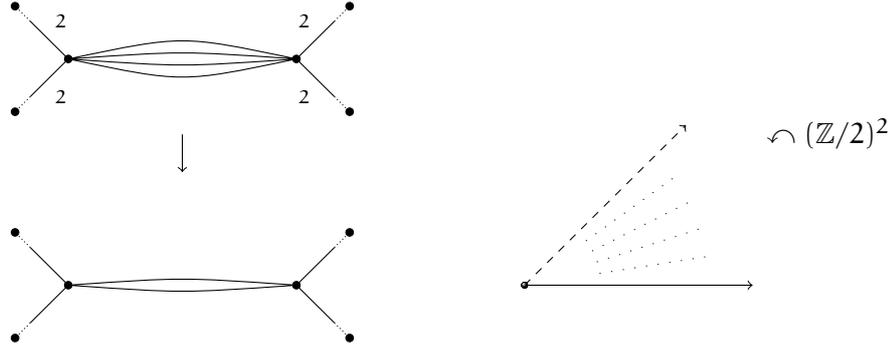
\begin{figure}[tb]
\begin{tikzpicture}
\tikzset{me/.style={to path={
\pgfextra{%
 \pgfmathsetmacro{\startf}{-(#1-1)/2}  
 \pgfmathsetmacro{\endf}{-\startf} 
 \pgfmathsetmacro{\stepf}{\startf+1}}
 \ifnum 1=#1 -- (\tikztotarget)  \else
     let \p{mid}=($(\tikztostart)!0.5!(\tikztotarget)$) 
         in
\foreach \i in {\startf,\stepf,...,\endf}
    {%
     (\tikztostart) .. controls ($ (\p{mid})!\i*6pt!90:(\tikztotarget) $) .. (\tikztotarget)
      }
      \fi   
     \tikztonodes
}}}   
\node[circle,fill=black,inner sep=1pt,draw] (a) at (0,0) {};
\node[circle,fill=black,inner sep=1pt,draw] (b) at (3,0) {};
\node[circle,fill=black,inner sep=1pt,draw] (c) at (-0.7,0.7) {};
\node[circle,fill=black,inner sep=1pt,draw] (d) at (-0.7,-0.7) {};
\node[circle,fill=black,inner sep=1pt,draw] (c) at (3.7,0.7) {};
\node[circle,fill=black,inner sep=1pt,draw] (d) at (3.7,-0.7) {};

\draw node at (-0.1,0.5) {{\tiny$2$}};
\draw node at (-0.1,-0.5) {{\tiny$2$}};

\begin{scope}[shift = {(3.2,0)}]
\draw node at (-0.1,0.5) {{\tiny$2$}};
\draw node at (-0.1,-0.5) {{\tiny$2$}};
\end{scope}

\draw (0,0)--(-0.5,0.5);
\draw [densely dotted] (-0.5,0.5)--(-0.7,0.7);
\draw (0,0)--(-0.5,-0.5);
\draw [densely dotted] (-0.5,-0.5)--(-0.7,-0.7);

\draw (3,0)--(3.5,0.5);
\draw [densely dotted] (3.5,0.5)--(3.7,0.7);
\draw (3,0)--(3.5,-0.5);
\draw [densely dotted] (3.5,-0.5)--(3.7,-0.7);

\draw (a) edge[me=4] (b); 

\draw [->] (1.5,-1)--(1.5,-1.5);

\begin{scope}[shift ={(0,-3)}]
\node[circle,fill=black,inner sep=1pt,draw] (a) at (0,0) {};
\node[circle,fill=black,inner sep=1pt,draw] (b) at (3,0) {};
\node[circle,fill=black,inner sep=1pt,draw] (c) at (-0.7,0.7) {};
\node[circle,fill=black,inner sep=1pt,draw] (d) at (-0.7,-0.7) {};
\node[circle,fill=black,inner sep=1pt,draw] (c) at (3.7,0.7) {};
\node[circle,fill=black,inner sep=1pt,draw] (d) at (3.7,-0.7) {};

\draw (0,0)--(-0.5,0.5);
\draw [densely dotted] (-0.5,0.5)--(-0.7,0.7);
\draw (0,0)--(-0.5,-0.5);
\draw [densely dotted] (-0.5,-0.5)--(-0.7,-0.7);

\draw (3,0)--(3.5,0.5);
\draw [densely dotted] (3.5,0.5)--(3.7,0.7);
\draw (3,0)--(3.5,-0.5);
\draw [densely dotted] (3.5,-0.5)--(3.7,-0.7);

\draw (a) edge[me=2] (b); 
\end{scope}

\begin{scope}[shift={(6,-3)}]
\draw node at (4,2) {$\curvearrowleft (\mathbb Z/2)^2$};
\draw [->] (0,0)--(3,0);
\draw [dashed, ->] (0,0)--(45:3);
\draw [ball color=black] (0,0) circle (0.4mm);
\foreach \a in {9,18,27,36}
	\draw [loosely dotted] (\a:1)--(\a:2.5);

\end{scope}

\end{tikzpicture}   
\caption{A combinatorial type and its associated cone. Undecorated internal edges of the top graph all have expansion factor $1$.The $2-1$ map to the base identifies the top pair of interior edges, and the bottom pair of interior edges. There is an effective $\mathbb Z/2$ group of automorphism switching the two interior edges of the bottom graph, and a  $(\mathbb Z/2)^2$  worth of automorphisms acting trivially on the cone,  corresponding to switching pairs of interior edges in the top graph mapping to the same edge. The dashed line on the left indicates folding from the automorphisms on the target. }
\label{fig: z/2-cone}
\end{figure}

\subsubsection{Automorphisms and weights on combinatorial types.} A stable tropical curve with no loop edges  generically has no automorphisms, since generic edge lengths are distinct, and automorphisms are required to be isometries. However, an admissible cover may have automorphisms generically, and act trivially on a generalized cone of the moduli space. Take for instance the cover depicted in Figure~\ref{fig: z/2-cone}. These automorphisms become relevant when we extract enumerative information from the degree of the branch map. 

\begin{remark}
Recall that such automorphisms are familiar in the classical setting. The hyperelliptic locus of $\Mbar_g$ can be understood as the space $\Hbar_{g\to 0,2}((2),\ldots,(2))$, which is the stack quotient of $\Mbar_{0,2g+2}$ by the trivial $\ZZ/2$ action. Here, the $\ZZ/2$ is naturally seen as acting on the covering curve, c.f. Figure~\ref{fig: aut-of-cover}. 
\end{remark}

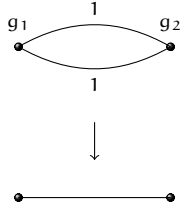
\begin{figure}[bt]
\begin{tikzpicture}
\draw [ball color=black] (0,0) circle (0.5mm);
\draw [ball color=black] (2,0) circle (0.5mm);
\path (0,0) edge [bend left] (2,0);
\path (0,0) edge [bend right] (2,0);
\draw node at (0,0.3) {\tiny$g_1$};
\draw node at (2,0.3) {\tiny$g_2$};
\draw node at (1,0.5) {\tiny$1$};
\draw node at (1,-0.5) {\tiny$1$};

\draw [ball color=black] (0,-2) circle (0.5mm);
\draw [ball color=black] (2,-2) circle (0.5mm);
\draw (0,-2)--(2,-2);

\draw [->] (1,-1)--(1,-1.5);
\end{tikzpicture}
\caption{A degree $2$ cover. The map identifies both finite edges in the source to the unique  edge in the target. The involution acts trivially on the one dimensional cone parameterizing metrization of the target edge.}
\label{fig: aut-of-cover}
\end{figure}

\begin{remark}{\bf (Interpretation of the compactification)} A compactification $\Mbar_{g,n}^{trop}$ is obtained from $\cM^{trop}_{g,n}$ by allowing edge lengths of interior edges to become infinity. This idea has a nice interpretation in terms of analytifications of curves. A skeleton of a curve over a valued field which has an internal edge of infinite length corresponds to a marked semistable model for that curve which has nodal generic fiber. In fact, the generic fiber of a semistable model is singular if and only if the associated skeleton has an infinite internal edge. See~\cite[Section 5]{BPR} and~\cite[Section 1.4]{Ber90} for details. 

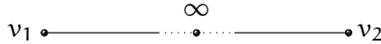
\begin{figure}[h!]
\begin{tikzpicture}
\draw node at (0,0.3) {\small$\infty$};
\draw node at (-2.3,0) {\small $v_1$};
\draw node at (2.3,0) {\small $v_2$};
\draw [ball color = black] (-2,0) circle (0.4mm);
\draw [ball color = black] (2,0) circle (0.4mm);
\draw [ball color = black] (0,0) circle (0.4mm);
\draw (-2,0)--(-0.5,0);
\draw [dotted] (-0.5,0)--(0.5,0);
\draw (0.5,0)--(2,0);
\end{tikzpicture}
\caption{An interior infinite edge, topologized as two infinite edges with the points at infinity identified.}
\end{figure}

In identical fashion, we obtain a compactification of the tropical Hurwitz space of tropical admissible covers, by allowing edges of the base, and the edges mapping to it, to tend to infinity. This is precisely the space $\Hbar^{trop}_{g,d}(\vec\mu)$. 
\end{remark}

The construction of Abramovich, Caporaso and Payne involves taking a colimit of cone complexes in the category of topological spaces and \textit{not} in the category of topological stacks. As a consequence, we need to explicitly remember the data of stabilizers in our enumerative calculations. 
\begin{definition}\label{wetheta}
Let $\Theta$ be a combinatorial type. We define its weight $\omega(\Theta)$ as the product of:
\begin{enumerate}[(W1)]
\item A factor of $\frac{1}{|Aut_0(\Theta)|}$.
\item A factor of local Hurwitz numbers $\prod_{v\in \Gamma_{tgt}} H(v)$.
\item A factor of $M = \prod_{e\in E(\Gamma_{tgt})} M_e$, where $M_e$ is the product of the expansion factors above the edge $e$, divided by their LCM.
\end{enumerate}
\end{definition}

\begin{remark} With  apologies for some unavoidable forward-referencing, let us briefly motivate where these weight factors arise in our theory. While $(W1)-(W3)$ are defined in
terms of combinatorics of the tropical covers, they have natural counterparts in the classical theory of admissible covers. 
Weight (W1) accounts for automorphisms of covers lifting the identity map on the target curve. We emphasize that by construction, these automorphisms act trivially on the space of tropical covers.
Term (W2) encodes the fact that there may be multiple zero dimensional strata in $\Hbar_{g,d}(\vec \mu)$ which have the same dual graph. 
The map 
\[
trop_\Sigma:{\Sigma}(\Hbar^{an}_{g\to h,d}(\vec\mu))\to \cH^{trop}_{g \to h,d}(\vec\mu),
\] 
defined in Theorem~\ref{thm: theorem1} identifies the distinct generalized cones of $\Sigma(\Hbar^{an}_{g,d}(\vec\mu))$ with a given dual graph to a single generalized cone in $\cH^{trop}_{g,d}(\vec\mu)$. See Section~\ref{sec: trop-sigma}.

Finally (W3) can be thought of either as ``ghost  automorphisms" coming from the orbifold structure on a twisted cover~\cite{ACV}. That is, it accounts for the fact that there one may place orbifold structure on an untwisted admissible cover in different ways. It may also be seen as arising from the normalization of the Harris--Mumford admissible cover space. We discuss this further in Section~\ref{sec: trop-sigma},  after discussing the deformation theory of $\Hbar_{g,d}(\vec\mu)$. (W3) is also  the generalization of the index of the matrix of ``length constraint equations'' studied in~\cite{BM13,CJM1}. 
\end{remark}

\subsection{Skeleta of toroidal embeddings} Associated to any toroidal embedding $U\hookrightarrow X$ is a cone complex with integral structure which we refer to as the skeleton $\Sigma(X)$. The idea is that locally analytically near a point $x$, $X$ looks like an affine toric variety $V_\sigma$. Thus, we can build a cone complex from these cones $\sigma$. The key difference between fans and abstract cone complexes is that abstract cone complexes do not come with a natural embedding into a vector space. In the case that the toroidal embedding is a toric variety, this cone complex is precisely the fan. It is worth observing though that unlike a toric variety, where the fan determines the variety, the cone complex $\Sigma(X)$ is far from determining $X$. 

For our purposes, the most important example of a toroidal embedding is the inclusion of the complement of a divisor with normal crossings, and we now explore this. Let $X$ be a normal scheme of dimension $n$. If $U\hookrightarrow X$ is given by the complement of a divisor with (not necessarily strict) normal crossings, then there is a natural stratification on $X$. That is, the $0$-strata are the $n$-fold intersections of divisors $D_i$, the $1$-strata are the $(n-1)$-fold intersections and so on. The top dimensional stratum is $U$. Consider a zero stratum $x\in X$. A formal neighborhood of $x$ looks like $n$ hyperplanes meeting at $x$. Locally near $x$, up to scaling, we obtain defining equations of these hyperplanes, say $f_1,\ldots, f_n$. These equations yield a system of formal local monomial coordinates near $x$. The completion of the local ring at $x$ is the coordinate ring of a formal affine space. The cone associated to this point is the standard cone for  the toric variety $\mathbb A^n$. Call this cone $\sigma$. For a $1$-stratum $W$, we get $(n-1)$ defining equations, giving a formal system of coordinates for an $(n-1)$ dimensional affine space. The cone associated to $W$ is the standard cone for $\mathbb A^{n-1}$. Call this cone $\tau$. Moreover, if $x\in W$, then the associated cone complex naturally identifies $\tau$ as a face of $\sigma$. This construction generalizes in the natural way, and the cones assemble into a cone complex $\Sigma(X)$. This yields an order reversing bijection between strata of the toroidal scheme $X$ and cones of the cone complex $\Sigma(X)$.

In~\cite{Thu07}, Thuillier shows that this cone complex lives naturally inside the Berkovich analytification of $X$. More precisely, given a toroidal embedding $U\hookrightarrow X$ with $X$ proper, he constructs a continuous (non-analytic) self-map of the Berkovich analytification of $X$, 
\[
\bm{p}_X:X^{an}\to X^{an}.
\] 
For non-proper but separated $X$, the identical statement holds, upon replacing $X^{an}$ with Thuillier's analytic formal fiber $X^\beth$. 

\begin{definition}\label{def: proj-to-skeleton}
The image of $\bm{p}_X$ is the \textit{skeleton} of $X$, denoted $\Sigma(X)$. The map $\bm p_X$ is referred to as the \textit{retraction to the skeleton}.
\end{definition}

The crucial fact for our purposes is the \textit{existence} of such a map and its properties. For an explicit realization in coordinates, see~\cite[Section~5.2]{ACP}. Abramovich, Caporaso, and Payne extend this construction to toroidal compactifications of Deligne--Mumford stacks, and produce a generalized (extended) cone complex and a retraction from the Berkovich analytification. We briefly describe the construction of the cone complex here. A detailed discussion of this retraction map in the setting of log structures may also be found in~\cite{U13}.

\subsubsection{Local toric models} A toroidal scheme $U\hookrightarrow X$ is described in a formal neighborhood of every point $x\in X$ by a toric chart $V_\sigma$. The cone $\sigma$ is described as follows. Let $M$ be the group of Cartier divisors supported on the complement of $U$. Let $M^+$ be the submonoid of effective Cartier divisors. Then the cone $\sigma$ is identified with the space of homomorphisms to the (additively written) monoid $\RR_{\geq 0}$,
\[
Hom(M^+,\RR_{\geq 0}),
\]
equipped with the natural structure of a rational polyhedral cone with integral structure. In the language of logarithmic geometry, sheafifying $M^+$ produces the \textit{characteristic monoid sheaf}, and $M$ produces the \textit{characteristic abelian sheaf.} The connections between tropical geometry and log geometry have been explored by Gross and Siebert~\cite[Appendix B]{GS13}, and Ulirsch, see~\cite{U13}.

\section{Tropicalization of the moduli space of admissible covers}\label{sec: tropicalizing-admissible-covers}

\subsection{Abstract tropicalization for admissible covers}\label{sec: metrization}
In this section we describe an abstract tropicalization for admissible covers.  The following tropicalization map --- which we denote $trop$, is obtained in direct analogy with the moduli space of curves, discussed in~\cite{ACP}. 

Let $\Hbar^{an}_{g,d}(\vec\mu)$  denote the Berkovich analytification of $\Hbar^{}_{g,d}(\vec\mu)$. A point $[D\to C]$ of $\Hbar^{an}_{g,d}(\vec\mu)$ is represented by an admissible cover over $Spec(K)$, where $K$ is a valued field extension of $\CC$. By properness of the stack $\Hbar^{}_{g,d}(\vec\mu)$, after ramified base change, the map extends uniquely to a family of curves over $Spec(R)$, where $R$ is a rank 1 valuation ring with valuation $val(-)$. Let $[\Gamma_D\to\Gamma_C]$ be the associated morphism of dual graphs of the special fibers. This morphism is well defined by the axioms placed on admissible covers. The ramification data of the admissible cover determines the expansion factors on all edges, therefore, we obtain a tropical admissible cover by metrizing these dual graphs. Let $e$ be an edge of $\Gamma_{C}$ corresponding to a node $q$. Choose an \'etale neighborhood of the node at $q$. The local equation is given by  $x_1x_2 = f$. We metrize the edge $e$ as $[0,val(f)]$. Note that an edge is metrized with length $\infty$ when $f=0$. The analytification of the boundary of $\Hbar^{an}_{g,d}(\vec\mu)$ parametrizes families of covers over a valuation ring whose generic fiber is a map of nodal curves

\begin{definition}\label{def: trop}
Let $[D\to C]$ be a point of $\Hbar^{an}_{g,d}(\vec\mu)$. With the notation above, we define the map
\begin{eqnarray*}
trop: \Hbar^{an}_{g,d}(\vec\mu) &\to& \Hbar^{trop}_{g,d}(\vec\mu)\\
{[D\to C]}&\mapsto& [\Gamma_D\to \Gamma_C]. 
\end{eqnarray*}
The map $trop$ naturally restricts to the open moduli spaces to give a map $\cH^{an}_{g,d}(\vec\mu) \to \cH^{trop}_{g,d}(\vec\mu)$.
\end{definition}

Given a family of admissible covers $[\pi: D\to C]$ over $Spec(K)$ where $K$ is a rank-$1$ valued field extending $\CC$, we have defined the associated tropical admissible cover in terms of models. Instead, one may work with the analytic curves themselves. 

In what follows, we freely use the language of semistable vertex sets from~\cite{BPR}. Let $[\pi: D\to C]$ be as above, with associated tropical cover $[\pi^{trop}:\Gamma_D\to \Gamma_C]$.

\begin{proposition}
There exist embeddings $\sigma_D:\Gamma_D\hookrightarrow D^{an}$ and $\sigma_C: \Gamma_C\hookrightarrow C^{an}$ such that the image of $\sigma_D$ (resp. $\sigma_C$) is a strong deformation retract of the analytic curve $D^{an}$ (resp. $C^{an}$). Moreover, the restriction of $\pi^{an}:D^{an}\to C^{an}$ to $\Gamma_D$ coincides with the map $\pi^{trop}:\Gamma_D\to \Gamma_C$. That is, the following diagram commutes
\[
\begin{tikzcd}
\Gamma_D\arrow{r}{\sigma_D} \arrow{d}[swap]{\pi^{trop}} & D^{an} \arrow{d}{\pi^{an}}\\
\Gamma_C\arrow{r}[swap]{\sigma_C} & C^{an}
\end{tikzcd}
\]
\end{proposition}

\begin{proof}
After a ramified base change, $[\pi: D\to C]$ extends to an admissible cover of marked semistable $R$-curves $\mathscr D\to \mathscr C$. Here, $R$ is a rank-$1$ valuation ring with residue field $k = \CC$. From~\cite[Section 5]{BPR}, the marked models $\mathscr D$ and $\mathscr C$ determine semistable vertex decompositions for the analytic curves $D^{an}$ and $C^{an}$ respectively. Every component of $\mathscr D_k$ maps onto a component of $\mathscr C_k$. Since the ramification points are all marked, the result now follows from~\cite[Section 4]{ABBR}.
\end{proof}


Given a stable tropical curve $\Gamma$, there always exists a marked model $\mathscr C$ over a rank-$1$ valuation ring, such that the metrized dual graph of $\mathscr C$ is $\Gamma$. The situation for admissible covers is more complicated, but still controllable, as exhibited by the following proposition.

\begin{proposition}
Let $\Theta$ be a combinatorial type for a tropical admissible cover, and $\Gamma_{src}\to \Gamma_{tgt}$ be a metrization of $\Theta$, with $\Gamma_{tgt}$ a totally degenerate, trivalent tropical curve. Then, $\Gamma_{src}\to \Gamma_{tgt}$ arises as the skeleton of an admissible cover of curves if and only if for all $v\in \Gamma_{tgt}$, $H(v)\neq 0$.
\end{proposition}

\begin{proof}
It is clear that there exists an admissible cover over $\CC$ whose dual graph has combinatorial type $\Theta$ precisely when $H(v)\neq 0$. The rest follows from the simultaneous smoothing theorem for metrized complexes in~\cite[Theorem B]{ABBR}. 
\end{proof}

Thus, if all ramification is marked, and the 
target graph is trivalent and totally degenerate, these local Hurwitz numbers encode the number of ways in which a morphism of metric graphs may be promoted to a morphism of nodal curves (or metrized complexes).

\subsection{Functorial tropicalization for the stack of admissible covers}
In~\cite{ACP}, it is shown that there is a generalized extended cone complex that is functorially associated to any toroidal compactification of a Deligne--Mumford stack, which lives as a retract of the Berkovich analytification. To describe the construction in our specific case, we first recall some facts about the deformation theory of admissible covers. Along the way we gather facts which will be useful in studying the tautological maps on $\Hbar^{trop}_{g\to h,d}(\vec\mu)$. 

\subsubsection{Deformation spaces I: An example}\label{sec: def1}

We analyze the completed local rings of points of the  moduli space $\Hbar_{g\to h, d}(\vec\mu)$
by explicitly normalizing the local rings of the Harris--Mumford stack $\HMbar_{g\to h,d}(\vec\mu)$.

To aid the reader, we begin with a toy example that illustrates the key features of the general case. Let $[D\to C]$ be an admissible cover in $\HMbar_{g\to h,d}(\vec\mu)$. Assume that $[D\to C]$ has no automorphisms. Let $z$ be a node of $C$ and  assume that there are two nodes $\widetilde z_1,\widetilde z_2$ above $z$, with ramification $2$ and $3$ respectively. Let $\xi$ be the deformation parameter of the node $z$, and let $\widetilde \xi_1$ and $\widetilde \xi_2$ be the deformation parameters of the nodes $\widetilde z_1,\widetilde z_2$. The situation is depicted in terms of dual graphs in Figure~\ref{fig: def-pic}.

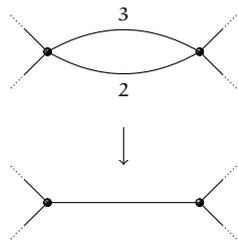
\begin{figure}[tb]
\begin{tikzpicture}
\draw [ball color=black] (0,0) circle (0.5mm);
\draw [ball color=black] (2,0) circle (0.5mm);
\path (0,0) edge [bend left] (2,0);
\path (0,0) edge [bend right] (2,0);
\draw node at (1,0.5) {\tiny$3$};
\draw node at (1,-0.5) {\tiny$2$};

\draw [ball color=black] (0,-2) circle (0.5mm);
\draw [ball color=black] (2,-2) circle (0.5mm);
\draw (0,-2)--(2,-2);

\draw (0,0)--(-0.3,0.3);
\draw [densely dotted] (-0.3,0.3)--(-0.5,0.5);
\draw (0,0)--(-0.3,-0.3);
\draw [densely dotted] (-0.3,-0.3)--(-0.5,-0.5);
\draw (2,0)--(2.3,0.3);
\draw [densely dotted] (2.3,0.3)--(2.5,0.5);
\draw (2,0)--(2.3,-0.3);
\draw [densely dotted] (2.3,-0.3)--(2.5,-0.5);

\draw (0,-2)--(-0.3,-2.3);
\draw (0,-2)--(-0.3,-1.7);
\draw [densely dotted] (-0.3,-2.3)--(-0.5,-2.5);
\draw [densely dotted] (-0.3,-1.7)--(-0.5,-1.5);
\draw (2,-2)--(2.3,-2.3);
\draw (2,-2)--(2.3,-1.7);
\draw [densely dotted] (2.3,-2.3)--(2.5,-2.5);
\draw [densely dotted] (2.3,-1.7)--(2.5,-1.5);

\draw [->] (1,-1)--(1,-1.5);
\end{tikzpicture}
\caption{The dual graph of the local picture at a node with ramification profile $(2,3)$ over it.}
\label{fig: def-pic}
\end{figure}

As we deform $z$, we need to deform $\widetilde z_1$ and $\widetilde z_2$ in accordance with the ramification profiles. Thus, the coordinate ring of the versal deformation space is 
\[
\CC\llbracket \xi,\widetilde \xi_1, \widetilde \xi_2\rrbracket/(\xi-\widetilde \xi_1^2,\xi-\widetilde \xi_2^3)\cong \CC\llbracket\widetilde \xi_1,\widetilde \xi_2\rrbracket/(\widetilde \xi_1^2-\widetilde \xi_2^3). 
\]

This is the completed local ring of a cuspidal cubic, at the cusp, and it is not integrally closed. Its integral closure is given by $\CC\llbracket\zeta\rrbracket$. The normalization map is given by $\widetilde \xi_1\mapsto \zeta^3$ and $\widetilde \xi_2\mapsto \zeta^2$. Thus, there is a unique point in the normalization $\Hbar_{g\to h,d}(\vec\mu)$ lying above $[D\to C]$. In particular, notice that $\xi\mapsto \zeta^6$. 

If we replace the ramification index $3$ above with $4$, then the completed local ring is the completed local ring at $(0,0)$ of two parabolas meeting at the origin in $\mathbb A^2$. That is, the completed local ring at $(0,0)$ of the affine place curve $V(\widetilde \xi_1^4-\widetilde \xi_2^2)$. In particular, notice that these completed local rings not only fail to be integrally closed, but even fail to be an integral domain. In this case the normalization is $\CC\llbracket\zeta\rrbracket\times \CC\llbracket\zeta\rrbracket$, and $\xi\mapsto (\zeta^4,\zeta^4)$. 

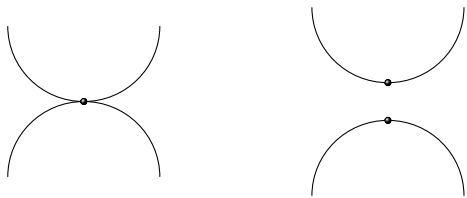
\begin{figure}[h!]
\begin{tikzpicture}
\draw (0,1) arc [radius = 1, start angle = 180, end angle = 360];
\draw (0,-1) arc [radius = 1, start angle = 180, end angle = 0];
\draw [ball color=black] (1,0) circle (0.4mm);

\begin{scope}[shift = {(4,0)}]
\draw (0,1.25) arc [radius = 1, start angle = 180, end angle = 360];
\draw (0,-1.25) arc [radius = 1, start angle = 180, end angle = 0];
\draw [ball color=black] (1,.25) circle (0.4mm);
\draw [ball color=black] (1,-.25) circle (0.4mm);
\end{scope}
\end{tikzpicture}
\label{fig: normalization}
\caption{The completed local ring of $\HMbar_{g,d}(\vec\mu)$ and its normalization for the $(2,4)$ ramification discussed above.}
\end{figure}

\subsubsection{Deformation spaces II: The general case} \label{sec: def2} We now tackle the general case, which is essentially the same as the example above, with certain clerical difficulties.

Consider a point $[\pi: D\to C]\in \HMbar_{g\to h,d}(\vec\mu)$ and let $[\Gamma_{src}\to\Gamma_{tgt}]$ be the dual combinatorial type of tropical admissible covers. Let $z_1,\ldots,z_k$ be the nodes of $C$, and let $\pi^{-1}(z_i) = \{z_{i,1},\ldots,z_{i,r(i)}\}$ be the preimages of the $i$th node. Assume that the ramification at $z_{i,j}$ is given by $p(i,j)$. Again assume  that $[D\to C]$ has no automorphisms. As argued in~\cite{HM82}, the completed local ring of $\HMbar_{g\to h,d}(\vec\mu)$ is given by 
\[
\widehat\cO_{[D\to C]} = \CC\llbracket\xi_1,\ldots,\xi_{N},\xi_{i,j}\rrbracket/(\xi_{i,j}^{p(i,j)}-\xi_i).
\]
Here, the first $k$ $\xi_i$ correspond to deformation parameters of the nodes of $C$, and the remaining are the parameters for deformations of complex structure on $C$. Our task is to compute the number of branches of the normalization of $\widehat\cO_{[D\to C]}$. 

Observe that the equations $(\xi_{i,j}^{p(i,j)}-\xi_i)$ involve the deformation parameters of a node $z_i$ of the target curve $C$ and those of the nodes of $D$ mapping to $z_i$, so may work node by node. 

In other words, the number of branches in the normalization of $\widehat\cO_{[D\to C]}$ is the product over nodes $z_i$ of $C$ of the number of branches in the normalization of
\[
\CC\llbracket \xi_i, \xi_{i,j}\rrbracket/(\xi_{i,j}^{p(i,j)}-\xi_i).
\]

To reduce the burden of notation, we drop subscripts: let $z$ be a node of $C$, and $\widetilde z_1,\ldots, \widetilde z_r$ be the nodes of $D$ mapping to $C$. Assume that the ramification order at $z_j$ is given by $p(j)$. Consider the ring
\[
R = \CC\llbracket \xi, \widetilde \xi_1,\ldots, \widetilde \xi_r \rrbracket/(\xi = \widetilde \xi_1^{p(1)} = \cdots =  \widetilde \xi_r^{p(r)}).
\]
The scheme $\spec(R)$ is a singular curve, and the normalization $\pi: \spec(\widetilde R)\to \spec(R)$ consists of a disjoint union of a number of copies of $\spec(\CC\llbracket s\rrbracket)$. The number of such copies is precisely the number of analytic branches we wish to compute. Let $L$ be the LCM of the orders of ramification $p(j)$ taken over the nodes $\widetilde z_j$ mapping to $z$. Upon restriction to an irreducible component of $\spec(\widetilde R)$, $\pi$ must be dual to a map of rings
\[
\pi^\#: \CC\llbracket \xi, \widetilde \xi_1,\ldots, \widetilde \xi_r \rrbracket/(\xi = \widetilde \xi_1^{p(1)} = \cdots =  \widetilde \xi_r^{p(r)})\to \CC\llbracket s\rrbracket,
\]
such that $\pi^\#(\xi) = s^L$. In order to satisfy the equations amongst the parameters $\widetilde \xi_j^{p(j)}$ we are forced to set $\pi^\#(\widetilde \xi_j) = \omega \cdot s^{L/p(j)}$, where $\omega$ is a $p(j)^{\mathrm{th}}$ root of unity. For each given node there are $p(j)$ such choices of roots of unity. Furthermore, there is a simultaneous scaling of the coordinates $\xi$ and $\widetilde \xi_j$ by a choice of $L^{\mathrm{th}}$ root of unity, so the total number of isomorphism classes of maps $\pi^\#$ satisfying $\pi^\#(\xi) = s^L$ is precisely $\prod_j p(j)/L$. Recombining these computations over all nodes of $C$ we see that the normalization of $\cO_{[D\to C]}$ is
\[
\prod_{i=1}^M \CC\llbracket\zeta_1,\ldots,\zeta_{N}\rrbracket,
\]
where $M = \prod_{e\in E(\Gamma_{tgt})} M_e$, and $M_e$ is the product of the ramification indices above the node corresponding to $e$, divided by their LCM. This gives us the desired explicit formula for how the deformation parameters for the nodes of $C$ pull back to the deformation parameters of the corresponding node in $[D\to C]$. 

A main technical result of~\cite{ACP} is that the skeleton of a toroidal Deligne--Mumford stack decomposes as a disjoint union of extended open cones corresponding to each stratum of the stack, modulo the appropriate monodromy. 

Recall from Definition~\ref{def: proj-to-skeleton} that a point of the skeleton $\Sigma(\Hbar^{an}_{g,d}(\vec\mu))$ is a point in the image of the retraction map $\bm p_{\Hbar}:\cH^{an}_{g,d}(\vec\mu)\to \cH^{an}_{g,d}(\vec\mu)$. Thus, given a point of the skeleton, represented by an admissible cover $[D\to C]$ over a valued extension field of $\CC$, we obtain, by taking dual graphs of the special fiber, a point $[\Gamma_D\to\Gamma_C]$ of the space $\cH^{trop}_{g,d}(\vec\mu)$. 

\begin{definition}\label{def: trop-sigma}
The map $trop_\Sigma: \Sigma(\Hbar^{an}_{g,d}(\vec\mu)) \to \cH^{trop}_{g,d}(\vec\mu)$ is defined to be the restriction of $trop$ to the skeleton.
As above, it sends an admissible cover $[D\to C]$ to the tropical cover $[\Gamma_D\to \Gamma_C]$.
\end{definition}

\subsubsection{Proof of Theorem~\ref{thm: theorem1}} \

Let $p\in \Hbar_{g,d}(\vec\mu)$ be a point contained in a locally closed stratum $W$ corresponding to a cover $[D\to C]$. Assume that the generic point of $W$ parametrizes a cover with combinatorial type $\Theta$. The point $p$ has an \'etale neighborhood $V\to \Hbar_{g,d}(\vec\mu)$, such that the locus parametrizing deformations of $[D\to C]$ where the $i^{\mathrm{th}}$ node of $C$ persists is a smooth and irreducible principal divisor cut out by a function $f_i$. Ranging over the nodes of $C$, we obtain a collection of monomials $f_1,\ldots, f_k$ in bijection with the nodes of $C$. The skeleton of $V^{\beth}$ is a copy of the extended cone $\overline \sigma_\Theta$ and the retraction
\[
V^\beth\to \overline\sigma_\Theta
\] 
is given by sending a valuation $\nu$ to the tuple $(\nu(f_1),\ldots, \nu(f_k))$. Upon application of~\cite[Proposition 6.2.6]{ACP}, the image of $\overline \sigma_\Theta$ in $\overline \Sigma(\Hbar^{an}_{g,d}(\vec\mu))$ is the quotient of the relative interior $\overline \sigma_\Theta^\circ$ of $\overline \sigma_\Theta$ by $Aut(\Theta)$. As a consequence we obtain a decomposition of the skeleton
\[
\overline \Sigma(\Hbar_{g,d}(\vec\mu)) = \bigsqcup_W \overline \sigma_\Theta^\circ/Aut(\Theta),
\]
where the disjoint union is taken over the locally closed strata $W$ of $\Hbar_{g,d}(\vec\mu)$. Note that unlike the case of $\Mbar_{g,n}$, there are multiple zero-dimensional strata of $\Hbar_{g,d}(\vec\mu)$ that have the same underlying combinatorial type.

The tropical moduli space similarly decomposes as
\[
\overline \Sigma(\Hbar_{g,d}(\vec\mu)) = \bigsqcup_\Theta \overline \sigma_\Theta^\circ/Aut(\Theta),
\] 
and thus we obtain a well-defined map of generalized extended cone complexes
\[
\Phi: \overline \Sigma(\Hbar_{g,d}(\vec\mu)) \to \Hbar^{trop}_{g,d}(\vec\mu).
\]
Given a locally closed stratum $W$ whose generic point parametrizes a combinatorial type $\Theta$, the map $\Phi$ simply identifies the generalized cone associated to $W$ with the generalized cone in $\Hbar^{trop}_{g,d}(\vec\mu)$ associated to $\Theta$.
The map $\Phi$ is manifestly continuous and surjective, and an isomorphism upon restriction to any generalized cone of the skeleton. We claim that this map $\Phi$ coincides with the map $trop_\Sigma$. To see this, choose a point $p'$ of $\Hbar^{an}_{g,d}(\vec\mu)$ such that the reduction of $p'$ is $p$. Locally, such a point gives rise to a valuation
\[
\nu: \CC[V]\to \RR\sqcup \infty,
\]
and $trop_\Sigma$ is defined by evaluating $\nu $ at the monomials $f_i$. On the other hand, observe that the monomials $f_1,\ldots, f_k$ are precisely the deformation parameters of the universal base curve in an \'etale neighborhood $V$ of $p$, and their valuations give rise to the edge lengths on the dual graph of $C$ that define the map $trop$. As a result the map $\Phi$ coincides with the restriction of $trop$ to the skeleton, so we conclude that $\Phi = trop_\Sigma$. We have obtain the desired factorization

\[
\begin{tikzcd}
\Hbar^{an}_{g\to h, d}(\vec \mu) \arrow{rr}{trop} \arrow{dr}[swap]{\bm p_H} & &  \Hbar^{trop}_{g\to h,d}(\vec \mu) \\
& \overline\Sigma(\Hbar^{an}_{g\to h, d}(\vec \mu)) \arrow{ur}[swap]{trop_{\Sigma}}. &
\end{tikzcd}
\]

\qed \\

\subsubsection{The map from the skeleton to the tropical moduli space}~\label{sec: trop-sigma} The skeleton $\overline\Sigma(\Mbar^{an}_{g,n})$ is a subset of the analytic space $\Mbar^{an}_{g,n}$ and one may restrict the pointwise tropicalization map to this subset to obtain a map $\overline{\Sigma}(\Mbar_{g,n}^{an})\to \Mbar_{g,n}^{trop}$. One way in which to view~\cite[Theorem 1.2.1]{ACP} is that this natural map is an isomorphism. This is not the case for admissible covers. The map 
\[
trop_\Sigma:{\Sigma}(\Hbar^{an}_{g,d}(\vec\mu))\to \Hbar^{trop}_{g,d}(\vec\mu),
\]
does not distinguish between strata which are ``combinatorially indistinguishable''. Said more precisely, there are multiple zero dimensional strata in $\Hbar_{g,d}(\vec\mu)$ corresponding to the same combinatorial type $\Theta = [\Gamma_{src}\to\Gamma_{tgt}]$. However, this loss of information can be completely characterized. Distinct generalized cones in the skeleton become identified in $\Hbar^{trop}_{g,d}(\vec\mu)$ for the following reasons. 
\begin{enumerate}[(A)]
\item There are multiple lifts of a trivalent, totally degenerate combinatorial type $\Theta$ to an admissible cover of nodal curves $[C\to D]$,
\item A nodal admissible cover $[D\to C]$ is really an element of the Harris--Mumford moduli space $\HMbar_{g,d}(\vec\mu)$. As discussed in Section~\ref{sec: def1}, there are multiple points in the normalization $\Hbar_{g,d}(\vec\mu)$ that lie above $[D\to C]$. As we have seen in Section~\ref{sec: def1} , there are precisely $M$ such points, where $M = \prod_{e\in E(\Gamma_{tgt})} M_e$, and $M_e$ is the product of the ramification indices above the node corresponding to the edge $e$, divided by their LCM. 
\end{enumerate}

These explain the weights on the tropical moduli space described in Section~\ref{sec: contractions}. The product of the local Hurwitz numbers  from (A) above gives (W2). The weight (W3) on $\cH^{trop}_{g\to h,d}(\vec\mu)$ is an artifact of (B) above. In the following section, we will see both (A) and (B) as a correction to the dilation factor $d_\Theta(br^{trop})$ referred to in Theorem~\ref{thm:branchdegree}.

\begin{remark}
The tropical moduli space $\Hbar^{trop}_{g,d}(\vec \mu)$ is naturally a generalized extended cone complex, but as we have seen the map from the skeleton of the classical space collapses several cones onto one. One may ask if there is a natural space whose analytification admits a tropicalization map to $\cH^{trop}_{g,d}(\vec \mu)$. The discussion of deformation spaces in Section~\ref{sec: def2} implies that the space $\HMbar_{g\to h,d}(\vec\mu)$ is locally analytically isomorphic to a \textit{possibly non-normal} toric variety and hence has the structure of a logarithmic scheme that is fine but possibly not \textit{saturated}. Though tropicalization has not been explicitly studied in this setting, following~\cite{U13}, one expects a continuous map from $\HMbar^{an}_{g\to h,d}(\vec\mu)$ to the $\RR_{\geq 0}\cup \infty$ valued points of the Kato fan of $\HMbar_{g\to h,d}(\vec\mu)$, and that this extended cone complex essentially coincides with $\Hbar^{trop}_{g,d}(\vec \mu)$. 
\end{remark}

\section{Classical and tropical tautological maps}\label{sec: taut-maps}

\subsection{The  branch maps} The classical admissible cover space $\Hbar_{g,d}(\vec \mu)$ admits a branch map, recording the base curve, marked at its branch points:
\begin{eqnarray*}
br:\overline{\mathcal H}_{g,d}(\vec\mu)&\to& \Mbar_{h,r+s} \\
{[D\to C]}&\mapsto& [(C,p_1,\ldots, p_r,q_1\ldots, q_s)].
\end{eqnarray*}
The degree of this map is the corresponding Hurwitz number. Intuitively, fixing a point $p$ in $\Mbar_{h,r+s}$, the cardinality of the fiber over $p$ is the number of covers having given discrete data, $(d,g,h,\vec\mu)$ over the specific pointed stable curve $C_p$. 

The tropical branch map is defined similarly, taking values in the space of pointed tropical $(r+s)$--pointed curves.
\begin{eqnarray*}
br^{trop}:\overline{\mathcal H}^{trop}_{g,d}(\vec\mu)&\to& \Mbar^{trop}_{h,r+s} \\
{[\Gamma_{src}\to \Gamma_{tgt}]}&\mapsto& [\Gamma_{tgt}].
\end{eqnarray*}

\subsection{Tropicalizing the branch map}

We want to compare the classical and tropical branch maps in a functorial manner. The process of taking skeletons for toroidal Deligne--Mumford stacks is compatible with toroidal and subtoroidal morphisms. Recall that a toroidal morphism $\phi:\mathscr X\to \mathscr Y$ is a morphism such that for every point $x\in \mathscr X$, there exist compatible \'etale toric charts around $x$ and $\phi(x)$, on which the morphism is given by a dominant equivariant map of toric varieties. We remark that although this is not pursued in~\cite{ACP}, we can impose a weaker condition on morphisms between toroidal embedding for which there is an induced map on skeleta. 

\begin{definition}\label{def: loc-an-toric}
A morphism $\varphi: \mathscr{X}\to \mathscr Y$ of toroidal embeddings is said to be \textit{locally analytically toric} if for every point $x\in \mathscr X$, there exist \'etale charts around $x$ and $\phi(x)$ such $\phi$ is given by an equivariant toric morphism. That is, we have the following diagram
\[
\begin{tikzcd}
\mathscr V_\sigma \arrow{d} & \mathscr V_x \arrow{d} \arrow{l} \arrow{r} & \mathscr X \arrow{d} \\
\mathscr V_\tau & \mathscr V_{\varphi(x)} \arrow{l}\arrow{r} & \mathscr Y,
\end{tikzcd}
\]
where the top row is a local toric chart around $x$ in $\mathscr X$, and the bottom row is a local toric chart around $\varphi(x)$ in $\mathscr Y$. The leftmost vertical arrow is torus equivariant. 
\end{definition}

It is easy to see that any morphism that is locally analytically toric induces a map on skeleta, and the statements for toroidal and subtoroidal morphisms in~\cite{ACP} go through without any substantial changes, see Remark 5.3.2 of loc. cit.

\begin{lemma}
The branch map $br:\overline{\mathcal H}_{g,d}(\vec\mu)\to \Mbar_{h,r+s}$ is locally analytically toric, where $\Mbar_{h,r+s}$ is given the toroidal structure from the inclusion $\cM_{h,r+s}\hookrightarrow \Mbar_{h,r+s}$. 
\end{lemma}

\textit{Proof.} To prove this, we need to find compatible local toric models for the points $x\in \Hbar_{g,d}(\vec\mu)$ and $br(x)\in \Mbar_{h,r+s}$, such that $br$ is given by an equivariant locally analytic toric morphism. It is sufficient to check that locally analytically $br$ pulls back monomials to monomials. However, this is clear from the discussion of deformation theory in Section~\ref{sec: def2}: deformations of an admissible cover being controlled by deformations of the target curve amounts precisely to saying that the local affine models around a boundary point are identified via the branch map.
\qed \\

Identifying the skeleton of $\Mbar^{an}_{h,r+s}$ with the tropical moduli space $\Mbar^{trop}_{h,r+s}$, we immediately have the following consequence.

\begin{corollary}
The analytified branch map $br^{an}:\Hbar^{an}_{g\to h,d}(\vec\mu)\to\Mbar^{an}_{h,r+s}$ induces a map on skeleta, $\overline{\Sigma}(\Hbar^{an}_{g\to h,d}(\vec\mu))\to \Mbar^{trop}_{h,r+S}$.
\end{corollary}

\textit{Proof.} Locally analytically, monomials are pulled back to monomials. Thus, it follows that there is an induced map on each cone of $\overline{\Sigma}(\Hbar^{an}_{g\to h,d}(\vec\mu))$ to the skeleton $\Mbar^{trop}_{h,r+S}$. The fact that these maps glue to give a global map is straightforward. \qed \\

\noindent
\subsubsection{Proof of Theorem \ref{thm: theorem2}, part one: branch map}~\label{thm2partI} It is clear from the description of the abstract tropicalization map for covers and for curves, that the tropical and classical branch maps fit together in a commutative diagram:
\[
\begin{tikzcd}
\Hbar_{g,d}^{an}(\vec \mu) \arrow{r}{{trop}}\arrow[swap]{d}{br^{an}} & \Hbar^{trop}_{g,d}(\vec \mu) \arrow{d}{br^{trop}}\\
\Mbar^{an}_{h,r+s} \arrow[swap]{r}{{trop}} & \Mbar^{trop}_{h,r+s} \\
\end{tikzcd}
\]

We recall from~\cite[Section 6]{ACP}, taking skeletons is functorial for toroidal morphisms. We thus also obtain the following more detailed commutative diagram. The commutativity of the left square follows by functoriality of taking skeletons. We use~\cite[Theorem 1.2.1]{ACP}  to identify the skeleton of $\Mbar_{h,r+s}$ with the tropical moduli space $\Mbar^{trop}_{h,r+s}$. Theorem~\ref{thm: theorem1} asserts that $trop_\Sigma$ is an isomorphism on each cone. Thus we get an extension of left square to the full diagram below.

\[
\begin{tikzcd}
\Hbar^{an}_{g,d}(\vec \mu) \arrow{r}{\mathbf{p}} \arrow[swap]{d}{br^{an}} \arrow[bend left]{rr}{trop} & \overline{\Sigma} (\Hbar_{g,d}(\vec \mu)) \arrow{d}{\overline{\Sigma}(br)} \arrow{r}{trop_\Sigma}& \Hbar_{g,d}^{trop}(\vec \mu)\arrow[densely dotted]{d}\\
\Mbar^{an}_{h,r+s}  \arrow{r}[swap]{\mathbf{p}} \arrow[bend right,swap]{rr}{trop}& \overline{\Sigma}(\Mbar_{h,r+s}) \arrow{r}[swap]{\psi}& \Mbar_{h,r+s}^{trop}
\end{tikzcd}
\]

The tropical branch map $br^{trop}$ fills in the far right vertical arrow, making the entire diagram commute, since $\psi$ is an isomorphism, and $trop_\Sigma$ is an isomorphism when restricted to any cone. It follows that the tropical branch map is indeed the tropicalization of the classical branch map, as desired. \qed

\subsection{The  source maps} The ``classical''  source map  takes a cover $[D \to C]$ to its source curve $[D]\in \Mbar_{g,n}$ where $n$ is the number of smooth ramification points, which is equal to the sum of the lengths of the partitions $\mu^i$. Similarly, there is a  tropical source map taking $[\Gamma_{src}\to\Gamma_{tgt}]$ to the source graph $\Gamma_{src}$. We wish to show that the tropical source map is naturally identified with the tropicalization of the analytified source map.

\subsubsection{Proof of Theorem \ref{thm: theorem2}, part two: source map} We want to show that in an \'etale neighborhood of a point $x\in \Hbar_{g,d}(\vec \mu)$, the map $src$ is given by a toric morphism. Let $x=[D\to C]$, and let $[\Gamma_{src}\to \Gamma_{tgt}]$ be its (unmetrized) dual graph. In $\Hbar_{g,d}(\vec \mu)$, the local monomial coordinates are given by the deformation parameters of the nodes of $C$. Let us focus on a single node $p\in C$ and on $\widetilde p_1,\ldots,\widetilde p_m$ the nodes of $D$ mapping to $p$. The nodes 
$\widetilde p_i$
can be independently deformed, and their deformation parameters  $\xi_1,\ldots, \xi_m$ are the local monomial coordinates on $\Mbar_{g,n}$.
The deformations of $[D]$ which come from deformations of the map $[D\to C]$  satisfy the relations
\[
\xi_1^{w_1} = \xi_2^{w_2} = \cdots = \xi_m^{w_m},
\]

where $w_i$ is the ramification on the node of $[D]$ corresponding to $\widetilde p_i$. Thus, locally analytically, $src$ maps $\Hbar_{g,d}(\vec\mu)$ via a toric morphism. This map induces a map on skeleta. 

To metrize $[\Gamma_{src}\to \Gamma_{tgt}]$ the above equations yield length conditions on edges $\widetilde e_i$ of $\Gamma_{src}$ mapping to a fixed edge $e$ of $\Gamma_{tgt}$:
\[
w_1 \ell(\widetilde e_1) = w_2\ell(\widetilde e_2) = \cdots = w_m\ell(\widetilde e_m),
\]
where $w_i$ is the expansion factors along $\widetilde e_i$. This collection of linear conditions cuts out a subcone of $\Mbar^{trop}_{g,n}$. The result now follows from similar arguments to Section~\ref{thm2partI}.
\qed

\subsection{The degree of the branch map} 
In this section we prove Theorem~\ref{thm:branchdegree}, by extracting the degree of the branch morphism (the Hurwitz number) from the associated map on skeleta. We can  compute this degree over any point of the moduli space $\Mbar_{h,r+s}$. Loosely speaking, the proof proceeds by first choosing a maximally degenerate base curve corresponding to a zero-stratum of the moduli space of targets. The branch map is given \'etale locally by a locally toric morphism, whose degree can be computed in local analytic monomial coordinates. Since both the source and target of the branch maps are stacks, care is needed to make this precise. 

%
%

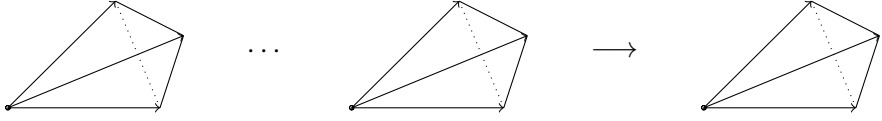
\begin{figure}[h!]
\begin{tikzpicture}
\matrix[column sep=0.7cm] {   
\begin{scope}[baseline]
\draw [ball color=black] (0,0) circle (0.3mm);
\draw [->] (0,0)--(2,0);
\draw [->] (0,0)--({sqrt(2)},{sqrt(2)});
\draw [->] (0,0)--(22.5:2.5);
\draw [dotted] ({sqrt(2)},{sqrt(2)})--(2,0);
\draw ({sqrt(2)},{sqrt(2)})--(22.5:2.5);
\draw (2,0)--(22.5:2.5);
\end{scope}
&
\begin{scope}[grow=right,baseline]
\draw node at (0,.75) {$\cdots$};
\end{scope}
&
\begin{scope}[grow=right,baseline]
\draw [ball color=black] (0,0) circle (0.3mm);
\draw [->] (0,0)--(2,0);
\draw [->] (0,0)--({sqrt(2)},{sqrt(2)});
\draw [->] (0,0)--(22.5:2.5);
\draw [dotted] ({sqrt(2)},{sqrt(2)})--(2,0);
\draw ({sqrt(2)},{sqrt(2)})--(22.5:2.5);
\draw (2,0)--(22.5:2.5);
\end{scope}
&
\begin{scope}[grow=right,baseline]
\draw node at (0,.75) {$\longrightarrow$};
\end{scope}
&
\begin{scope}[grow=right,baseline]
\draw [ball color=black] (0,0) circle (0.3mm);
\draw [->] (0,0)--(2,0);
\draw [->] (0,0)--({sqrt(2)},{sqrt(2)});
\draw [->] (0,0)--(22.5:2.5);
\draw [dotted] ({sqrt(2)},{sqrt(2)})--(2,0);
\draw ({sqrt(2)},{sqrt(2)})--(22.5:2.5);
\draw (2,0)--(22.5:2.5);
\end{scope}
\\};
\end{tikzpicture}
\caption{Cones in $\Sigma(\Hbar^{an}_{g,d}(\vec\mu))$ lying above a top dimension stratum in $\Mbar^{trop}_{h,r+s}$.}
\end{figure}

\subsubsection{Proof of Theorem~\ref{thm:branchdegree}} The proof proceeds in two steps. We first show that the degree of the branch map can be recovered from the maps on skeleta. We then proceed to show that with the weighting introduced in Section~\ref{sec: contractions}, the degree of the tropical branch map is equal to the degree of the map on skeleta. 

The branch morphism analytifies to a map $br^{an}:\Hbar_{g,d}^{an}(\vec\mu)\to \Mbar_{h,r+s}^{an}$ of analytic stacks, and we wish to compute the degree of this map. Let $\Gamma$ be a combinatorial graph of genus $h$, with $k = 3h-3+r+s$ internal edges. That is, $\Gamma$ is dual to a maximally degenerate curve $C$. This curve gives rise to a moduli map $\spec(\CC)\to \Mbar_{h,r+s}$ and hence a point $p$ of the stack. There is an \'etale neighborhood $U_\Gamma\to \Mbar_{h,r+s}$ of $p$ and an \'etale map $U_\Gamma\to \spec(\CC\llbracket \xi_1,\ldots, \xi_k\rrbracket)$. Note that the degree of the map $U_\Gamma\to \Mbar_{h,r+s}$ is precisely $|Aut(\Gamma)|$.  The open set $U_\Gamma$ is a versal deformation space for the curve dual to $\Gamma$, and $\xi_i$ is identified with the smoothing parameter for the node of $C$ corresponding to the $i^{\mathrm{th}}$ edge of $\Gamma$. The maximal dimensional cone $\sigma_{\Gamma}\cong \RR_{\geq 0}^k$ is canonically identified with the skeleton of the analytic space $U^\beth_\Gamma$. 

Choose an admissible cover $[D\to C]$, such that $C$ is dual to $\Gamma$. Fix an identification of the dual graph of $C$ with $\Gamma$. Let $\Theta$ be the associated combinatorial type for this cover and $\sigma_\Theta$ the moduli space of tropical admissible covers with type $\Theta$. The choice of identification of the nodes of $C$ with the edges of $\Gamma$ gives rise to a fiber diagram
\[
\begin{tikzcd}
U_\Theta \arrow[swap]{d}{\widetilde{br}_\Theta}\arrow{r} & \Hbar_{g,d}(\vec\mu) \arrow{d}{br} \\
U_\Gamma \arrow{r} & \Mbar_{h,r+s}.
\end{tikzcd}
\]
Here $U_\Theta$ is an \'etale neighborhood of $[D\to C]$. Note that the local coordinates given by $U_\Gamma$ are precisely the deformation parameters for the nodes of $C$, and thus give rise to local coordinates on $U_\Theta$. Since $U_\Gamma$ parametrizes singular curves with marked nodes, the only possible automorphisms on covers parametrized by $U_\Theta$ are those which act trivially on the base. That is, $U_\Theta$ is isomorphic to the stack quotient $[\A^k/Aut_0(\Theta)]$, where $Aut_0(\Theta)$ acts trivially on $\A^k$. Recall that $Aut_0(\Theta)$ consists of automorphisms of the cover $\Theta$ that lift the identity on $\Gamma$. After taking a further cover of $U_\Theta$ to account for this stabilizer, we see that the local degree of the map of affinoid domains
\[
\widetilde{br}^\beth_\Theta: U^\beth_\Theta\to U^\beth_\Gamma
\]
is given by $\frac{1}{|Aut_0(\Theta)|}$ times the degree of the map $\A_\beth^k\to \A_\beth^k$ induced by the map of cones $\sigma_\Theta\to \sigma_\Gamma$.

Consider the two projection maps $\bm p_H: U_\Theta^{\beth}\to \overline\sigma_\Theta$ and $\bm p_M: U^\beth_\Gamma\to \overline\sigma_\Gamma$. The inverse images of $\sigma_\Theta$ and $\sigma_\Gamma$ respectively give rise to polyhedral domains $\mathscr U_H$ and $\mathscr U_M$ in analytifications of formal tori. By~\cite[Section 6]{R12} the degree of the map $br^{an}$ restricted to $\mathscr U_H$ is given by the determinant of the morphism $\sigma_\Theta\to \sigma_\Gamma$. Furthermore, we have coordinates on these analytic tori, given by the deformation parameters, as previously discussed. Let $\xi_i$ be coordinates on $\mathscr U_M$ and $\widetilde \xi_i$ the coordinates on $\mathscr U_H$. It follows from the discussion in Section~\ref{sec: def1} that $\varphi^*(\xi_i) = \widetilde \xi_i^N$ where $N$ is the LCM of the ramification indices at the nodes lying over the $i^{\mathrm{th}}$ node. This is precisely equal to the dilation factor in the $i^{\mathrm{th}}$ coordinate for this map of covers of cones. Hence, locally analytically, the degree of this map is equal to the product of the LCM's of the ramifications over each node. 

We now need to pass from the skeleton $\Sigma(\Hbar^{an}_{g,d}(\vec\mu))$ to the tropical space $\cH^{trop}_{g,d}(\vec\mu)$.  Continue to fix the target graph $\Gamma$, and the combinatorial type $\Theta$ of the admissible cover. Choose an identification of the dual graph of the target with $\Gamma$. We need to understand how many points of $\Hbar_{g,d}(\vec\mu)$ have type $\Theta$. By choice of $\Gamma$, we can uniquely build a curve $C$, up to automorphisms of $\Gamma$, such that the dual graph of $C$ is $\Gamma$. The number of Harris--Mumford admissible covers such that the dual graph of the target is identified with $\Gamma$ is by definition the product of the local Hurwitz numbers of $\Theta$. Moreover, for each such algebraic admissible cover $[D\to C]$ in $\HMbar_{g,d}(\vec\mu)$, the number of preimages of $[D\to C]$ in $\Hbar_{g,d}(\vec\mu)$ is given by the weight (W3) in Section~\ref{sec: contractions}, namely, the product $\prod_{e\in E(\Gamma_{tgt})} M_e$, where $M_e$ is the product of the ramification indices above the node corresponding to $e$, divided by their LCM. These precisely are the weights on the tropical admissible cover space, and thus, we see that the weighted degree of the map $br^{trop}$ is equal to the degree of the map $br^\Sigma$. 

To conclude, observe that after a barycentric subdivision of source and target of $br^{trop}$, the branch map 
\[
br^{trop}: B(\cH^{trop}_{g,d}(\vec\mu))\to B(\cM^{trop}_{h,r+s})
\] 
is a morphism of cone complexes, and each cone of $B(\cH^{trop}_{g,d}(\vec\mu))$ maps isomorphically onto a an extended cone in the target, ignoring integral structures. Following the discussion in~\cite[Section 4.2]{ACP}, any top dimensional cone $\sigma$ of $B(\cM^{trop}_{h,r+s})$ has an associated combinatorial type $\Gamma$, and similarly, any top dimensional extended cone $ \sigma'$ of $B(\cH^{trop}_{g,d}(\vec\mu))$ has a combinatorial type $\Theta$. If $\sigma'$ lies over $\sigma$, then the base graph of $\Theta$ is identified with $\Gamma$. By the computation above, the weighted sum of dilation factors of cones of $B(\cH^{trop}_{g,d}(\vec\mu))$ mapping to $\overline \sigma$ precisely recovers the degree of $br^{an}$ and the result follows. 
\qed

\section{Applications to previous work}\label{sec: applications}

In this section, we recover known correspondence theorems for tropical Hurwitz numbers at the level of moduli spaces. In Section~\ref{sec: double-numbers}, we return to the motivating case of double Hurwitz numbers. The first correspondence theorem for double Hurwitz numbers was proved in~\cite{CJM1}. In that work, a tropical analogue of the relevant relative stable map space is constructed, in order to carry out the relevant intersection theory computations. We pay special attention to the relation of our tropical admissible cover spaces to the ones used there. The equality of tropical and classical double Hurwitz numbers can also be deduced from the general correspondence theorem of~\cite{BBM}, which we recover at the level of moduli spaces in Section~\ref{subsec-BBM}.

\subsection{Monodromy graphs and tropical double Hurwitz numbers}\label{sec: double-numbers}

We now frame the  ``monodromy graphs'' computation for the double Hurwitz numbers, introduced by the first two authors and Paul Johnson, in a geometric context. First, we briefly recall the relevant aspects of~\cite{CJM1}.

We fix two partitions $\mu^1 = (\mu^1_1,\ldots,\mu^1_k)$ and $\mu^2 = (\mu^2_1,\ldots,\mu^2_\ell)$ of degree $d$, and denote $s = 2g-2+\ell+k$, the number of simple branch points, determined by Riemann--Hurwitz.

\begin{definition}
\textit{Monodromy graphs} project to the segment $[0,s+1]$ and are constructed as follows:
\begin{enumerate}[(i)]
\item Start with $k$ small segments over $0$, with weights $\mu^1_1,\ldots, \mu^1_k$.
\item Over the point $1$, create a trivalent vertex by either joining two strands or splitting two strands. When joining two strands, label the outgoing edge with the sum of the incoming weights. In case of a cut, label the two new strands in all possible positive ways of adding the weight of the split edge. Each choice of split produces a distinct monodromy graph.
\item Repeat this process for integers up to $s$.
\item Retain all connected graphs that terminate with $\ell$ points of weights $\mu^2_1,\ldots,\mu^2_\ell$ over $s+1$.
\end{enumerate}
\end{definition}

It is proved in~\cite{CJM1} that such graphs produce a formula for the double Hurwitz number.

\begin{theorem}[C--Johnson--M]
The double Hurwitz number is equal to
\[
h_{g\to 0,d}(\vec\mu) = \sum_{\Gamma}\frac{1}{|Aut(\Gamma)|} \prod w(e),
\]
where we take the sum over isomorphism classes of monodromy graphs, and the product of interior edge weights of each graph (i.e. edges not over $0$ or $s+1$). 
\end{theorem}

\begin{example}
Consider for instance the monodromy graph depicted in Figure~\ref{fig: ex-monodromy-graph}. This graph has two automorphisms, coming from the double edge (``wiener").  We ignore the automorphisms coming from the ``fork" since the ends are now marked. The product of the weights of interior edges is $16$, so this graph contributes $8$ to the sum in the preceding theorem. 
\end{example}

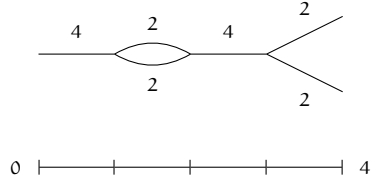
\begin{figure}[h!]
\begin{tikzpicture}
\draw (-1,0)--(0,0);
\path (0,0) edge [bend left] (1,0);
\path (0,0) edge [bend right] (1,0);
\draw (1,0)--(2,0);
\draw (2,0)--(3,0.5);
\draw (2,0)--(3,-0.5);

\draw [|-|] (-1,-1.5)--(0,-1.5);
\draw [-|] (0,-1.5)--(1,-1.5);
\draw [-|] (1,-1.5)--(2,-1.5);
\draw [-|] (2,-1.5)--(3,-1.5);

\draw node at (-0.5,0.3) {\tiny $4$};
\draw node at (0.5,0.4) {\tiny $2$};
\draw node at (0.5,-0.4) {\tiny $2$};
\draw node at (1.5,0.3) {\tiny $4$};
\draw node at (2.5,0.6) {\tiny $2$};
\draw node at (2.5,-0.6) {\tiny $2$};

\draw node at (-1.3,-1.5) {\tiny$ 0$};
\draw node at (3.3,-1.5) {\tiny$ 4$};

\end{tikzpicture}
\caption{A genus $1$ monodromy graph for degree $4$ covers of $\PP^1$ with ramification profiles $(4)$ over $0$ and $(2,2)$ over $\infty$.}
\label{fig: ex-monodromy-graph}
\end{figure}

We now study the monodromy graphs in terms of admissible covers, and recover the above formula. Moreover, we reinterpret the expected dimensional cells in the tropical moduli space of parametrized covers as a collection of cones in $\Hbar^{trop}_{g\to 0,d}(\mu^1,\mu^2)$, lying over a chosen cell in $\Mbar^{trop}_{0,2+s}$. From this vantage point, we see the factors appearing in the above sum from the geometric perspective developed in previous sections.  

\subsubsection{The CJM covers of $\PP^1_{trop}$} In order to reinterpret the CJM formula in our framework, we need to build tropical admissible covers from the tropical relative stable maps considered in~\cite{CJM1}. Not all relative stable maps produce admissible covers, but those combinatorial types of \textit{expected dimension} do produce admissible covers. Denote by $\PP^1_{trop}$, the ``two pointed'' tropical $\PP^1$, $\RR\cup\{\pm\infty\}$. We first recall the definition of tropical covers of $\PP^1_{trop}$ as stated in~\cite{CJM1}. 

\begin{definition}
Let $\mu^1,\mu^2$ be partitions of $d$. Let $\Gamma$ be a genus $g$, $\ell((\mu^1)+\ell(\mu^2))$-pointed tropical curve. A \textit{parametrized tropical curve of genus $g$ and degree $(\mu^1,\mu^2)$ in $\PP^1$} is a integral harmonic morphism $\theta: \Gamma\to \PP^1_{trop}$, where $\Gamma$ has genus $g$, such that 
\begin{enumerate}[(i)]
\item The image of $\Gamma$ without its infinite edges is inside $\RR$.
\item The multiset of expansion factors over the $+\infty$ segment is given by $\mu^1$, and the multiset of expansion factors over the $-\infty$ segment is given by $\mu^2$. 
\end{enumerate}
\end{definition}

A \textit{combinatorial type} of a parametrized tropical curve in $\PP^1$ is the data obtained from dropping the edge length data and remembering only the source curve together with its expansion factors. For a combinatorial type $[\alpha]$ of covers, we build an unbounded, open unbounded convex polyhedron, formed by varying edge lengths. These cells glue together to form a moduli space $M_{g}(\PP^1_{trop},\mu^1,\mu^2)$. We refer to~\cite{CJM1} for details on the construction. 

The key difference between parametrized tropical curves in $\PP^1_{trop}$ and admissible covers is that parametrized curves in $\PP^1_{trop}$ may contract subgraphs. However, if there is a combinatorial type with a contracted component, it will not be of expected dimension, and will not contribute to the degree. In fact, the moduli space of tropical covers constructed in~\cite{CJM1} does not consider cells where the associated combinatorial type is not of expected dimension.   

It is shown in~\cite{CJM1} that the moduli space admits a natural branch map to $(\PP^1_{trop})^s$ where $s$ is the number of simple branch points. Moreover, by weighting this moduli space appropriately, the degree of this branch map essentially recovers the formula above. In particular, the factor $\prod w(e)$ arises as a product of the determinant of the branch map, times a certain weight on each cone of $M_{g}(\PP^1_{trop},\mu^1,\mu^2)$. We remark that in the construction of $M_{g}(\PP^1_{trop},\mu^1,\mu^2)$, it is necessary to disregard cells of unexpected dimension, in order to obtain a well defined degree.

Given a parametrized tropical cover whose combinatorial type is of expected dimension, we obtain a admissible cover by first giving the base $\PP^1_{trop}$ the natural structure of a $2$-pointed tropical curve as follows. We mark the images of all branch points of $\Gamma$. Since the combinatorial type is of expected dimension, there are precisely $s$ points which are marked, where $s$ is the number of simple branch points.

Additionally, we subdivide $\Gamma$ such that vertices map to vertices. This amounts to making each point in the preimage of a branch point into a vertex, see Figure~\ref{fig: subdivision}. Finally, we add an infinite edge to each of the $s$ marked points on the base obtaining what we call the path graph on $s$ vertices. The ramification over these infinite edges is simple. By the local Riemann--Hurwitz condition, there is a unique way to add $(d-1)$ infinite edges mapping to each new infinite edge added on the base graph, such that the ramification over each infinite edge is simple. We record the following observation.

\begin{proposition}
Let $\Gamma$ be the path graph on $s$ vertices, and let $\cH_\Gamma$ be subcomplex of $\cH^{trop}_{g,d}(\mu^1,\mu^2)$ such that $\Gamma_{tgt} = \Gamma$. Then, there is an identification of cone complexes
\[
\cH_\Gamma \cong M_g(\PP^1_{trop},\mu^1,\mu^2).
\]
\end{proposition}

The proof follows immediately from the preceding discussion. 

\begin{figure}[h!]
\begin{tikzpicture}
\draw (-1,0)--(0,0);
\draw (0,0)--(30:3.3);
\draw (0,0)--(-30:1);
\path (-30:1) edge [bend left] (2,-0.5);
\path (-30:1) edge [bend right] (2,-0.5);
\draw (2,-0.5)--(3,-0.5);


\draw node at (-0.5,0.3) {\tiny $4$};
\draw node at (1.5,1.2) {\tiny $2$};
\draw node at (.4,-0.6) {\tiny $2$};
\draw node at (1.45,-0.9) {\tiny $1$};
\draw node at (1.45,-0.1) {\tiny $1$};
\draw node at (2.5,-0.3) {\tiny $2$};
\draw node at (-1.3,-1.5) {\tiny $-\infty$};
\draw node at (3.3,-1.5) {\tiny $+\infty$};

\draw [|-|]  (-1,-1.5)--(0,-1.5);
\draw [-|]  (0,-1.5)--(0.85,-1.5);
\draw [-|]  (0.85,-1.5)--(2.01,-1.5);
\draw [-|]  (2.01,-1.5)--(3,-1.5);

\draw [loosely dotted] (0.85,-1.5)--(0.85,0.5);
\draw [ball color=black] (0.85,0.5) circle (0.3mm);
\draw [loosely dotted] (2.01,-1.5)--(2.01,1);
\draw [ball color=black] (2.01,1.15) circle (0.3mm);

\begin{scope}[shift = {(7,0)}]
\draw (-1,0)--(0,0);
\draw (0,0)--(30:3.3);
\draw (0,0)--(-30:1);
\path (-30:1) edge [bend left] (2,-0.5);
\path (-30:1) edge [bend right] (2,-0.5);
\draw (2,-0.5)--(3,-0.5);

\draw (0,0)--(-0.5,-0.5);
\draw (0,0)--(-0.7,-0.5);
\draw (0,0)--(-0.3,-0.5);
\draw (0.85,0.5)--(0.35,0);
\draw (-30:1)--(0.4,-1);
\draw (-30:1)--(0.6,-1);
\draw (2.01,1.15)--(1.51,0.65);
\draw (2,-0.5)--(1.5,-1);
\draw (2,-0.5)--(1.7,-1);


\draw node at (-0.5,0.3) {\tiny $4$};
\draw node at (1.5,1.2) {\tiny $2$};
\draw node at (.4,-0.6) {\tiny $2$};
\draw node at (1.42,-0.9) {\tiny $1$};
\draw node at (1.45,-0.1) {\tiny $1$};
\draw node at (2.5,-0.3) {\tiny $2$};
\draw node at (-1.3,-1.5) {\tiny $-\infty$};
\draw node at (3.3,-1.5) {\tiny $+\infty$};

\draw [|-|]  (-1,-1.5)--(0,-1.5);
\draw [-|]  (0,-1.5)--(0.85,-1.5);
\draw [-|]  (0.85,-1.5)--(2.01,-1.5);
\draw [-|]  (2.01,-1.5)--(3,-1.5);

\draw [loosely dotted] (0.85,-1.5)--(0.85,0.5);
\draw [ball color=black] (0.85,0.5) circle (0.3mm);
\draw [loosely dotted] (2.01,-1.5)--(2.01,1);
\draw [ball color=black] (2.01,1.15) circle (0.3mm);

\draw(0,-1.5)--(-0.5,-2);
\draw (0.85,-1.5)--(0.35,-2);
\draw (2.01,-1.5)--(1.51,-2);
\end{scope}

\end{tikzpicture}
\caption{A parametrized tropical curve of genus $1$ on the left, and the corresponding admissible cover on the right. The infinite edges carry simple ramification $(2,1,\ldots,1)$. The expansion factors of $2$ occur on trivalent vertices of the parametrized tropical curve. }
\label{fig: subdivision}
\end{figure}
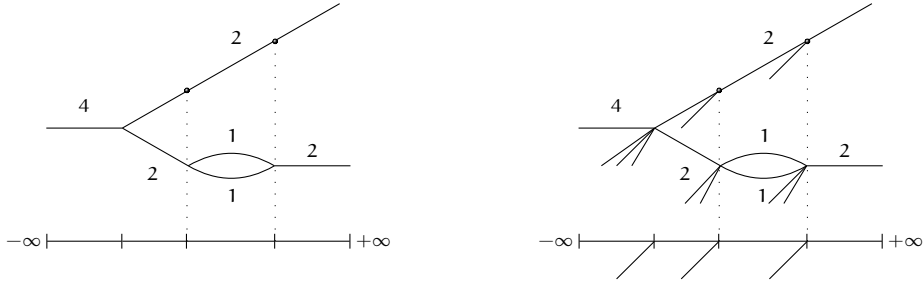

When covers have no contracted components, and the cell of the tropical moduli space of the expected dimension, we essentially recover a cone of tropical admissible covers. The only change is that here, we drop the data of the root vertex, which is irrelevant, since the combinatorics does not change if we translate the $s$ branch points by a fixed real number. It is easy to see that the branch map defined in~\cite{CJM1} transforms naturally into the map $br^{trop}$ defined here. With this translation, we recover the previously defined tropical double Hurwitz numbers, as we now demonstrate. 

\subsubsection{The CJM formula for the double Hurwitz number}~\label{sec: CJM}
We now turn our attention to the expression of double Hurwitz numbers in terms of monodromy graphs, studied in~\cite{CJM1}.

The following useful proposition is due to Lando and Zvonkine~\cite{LZ99}.
\begin{proposition}\label{prop: LZ}
Let $\mu^1 = (d)$ and $\mu^2$ be arbitrary with $t$ parts. Then we have the formula
\[
h_{0\to 0,d}(\mu^1,\mu^2) = (t-1)!d^{t-2}. 
\]
Observe that if $\mu^2$ is a \textit{two-part partition}, i.e. $t=2$, then $h_{0\to 0,d}(\mu^1,\mu^2) = 1$. 
\end{proposition}

\begin{warning}\label{warning: subdivision}
Given a parametrized tropical curve, when we subdivide, we create new vertices and consequently new interior edges. There is a unique expansion factor on his new interior edges by the harmonicity condition. Let $v$ be a new vertex created in such a manner by subdivision. The local Hurwitz number at $v$ is given by $h_{0\to 0,d}((d),(d))$, which we know to be $1/d$. It will be crucial in the forthcoming discussion that the quantity $h_{0\to 0,d}((d),(d))$, times the weight on the new bounded edge is  $1$. See Figure~\ref{fig: subdivision}. For a related issue, see the discussion in~\cite[Lemma 3.5]{BM13}. 
\end{warning}

\noindent
\textit{Proof of the CJM formula.} We work with the tropical admissible cover space $\overline{\cH}^{trop}_{g\to 0, d}(\mu^1,\mu^2)$. Since the degree of the tropical branch map is constant, we may choose to compute it over a fixed top dimensional cell in $\Mbar^{trop}_{0,2+s}$, where $s$ is the number of simple branch points. We choose the locus of curves whose combinatorial type is a path graph, augmented with one infinite edge at every bivalent vertex, as shown in Figure~\ref{fig: m0n-cell}. We denote this combinatorial type by $[\Gamma]$, and by $\cM_{\Gamma}$ the corresponding cell of the tropical moduli space. 

It was observed in~\cite[Remark~5.2]{CJM1} that the combinatorial types lying over $\cM_{\Gamma}$ have a totally degenerate genus function on the source curve. Furthermore, after contracting the infinite edges corresponding to the $s$ simple branch points and their preimages, every edge of $\Gamma_{src}$ is trivalent. Consequently we see that the profiles for the local Hurwitz numbers of $\Gamma_{src}$ are given by $(d)$ (total ramification) and a two-part ramification profile. Thus, all local Hurwitz numbers are $1$ by Proposition~\ref{prop: LZ}, except those introduced by subdivision, which were discussed in Warning~\ref{warning: subdivision}. 

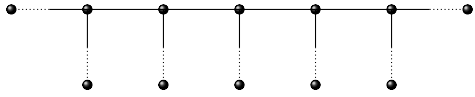
\begin{figure}[h!]
\begin{tikzpicture}[scale=2]
\draw [ball color=black] (1,0) circle (0.3mm);
\draw [ball color=black] (1.5,0) circle (0.3mm);
\draw [ball color=black] (2,0) circle (0.3mm);
\draw [ball color=black] (2.5,0) circle (0.3mm);
\draw [ball color=black] (3,0) circle (0.3mm);

\foreach \a in {1,1.5,2,2.5}
	\draw (\a,0)--(\a+0.5,0);
\draw  (0.75,0)--(1,0);
\draw  (3.25,0)--(3,0);
\draw [densely dotted] (3.25,0)--(3.5,0);
\draw [densely dotted] (0.75,0)--(0.5,0);
\draw [ball color=black] (3.5,0) circle (0.3mm);
\draw [ball color=black] (0.5,0) circle (0.3mm);

\foreach \a in {1,1.5,2,2.5,3}
	\draw (\a,0)--(\a,-0.25);
\foreach \a in {1,1.5,2,2.5,3}
	\draw [densely dotted] (\a,-0.25)--(\a,-0.5);
\foreach \a in {1,1.5,2,2.5,3}
	\draw [ball color=black] (\a,-0.5) circle (0.3mm);

\end{tikzpicture}
\caption{\small The chosen cell $\cM_\Gamma$ in $\Mbar^{trop}_{0,2+s}$.}
\label{fig: m0n-cell}
\end{figure}

The cones mapping onto $\M_{\Gamma}$ via the branch map are precisely those cones $\cH_{\Theta}$ where $\Theta = [\Gamma_{src}\to \Gamma]$. Here the $\Gamma_{src}$  precisely correspond to the monodromy graphs of~\cite{CJM1}. To compute the degree of the branch map over this cell $\cM_{\Gamma}$, we need only compute the degrees of maps from the individual top dimensional cells lying over $\cM_{\Gamma}$, and add the resulting contributions. 
%

Consider an admissible cover $[D\to C]$ in $\Hbar_{g\to 0,d}(\mu^1,\mu^2)$ lying over the stable nodal genus $0$ curve $[C]$ in $\Mbar_{0,2+s}$. We denote by $\xi_i$ the deformation parameter of the $i$th node of $[C]$, and by $\widetilde \xi_i$ the deformation parameter of the $i$th node of the base of the admissible cover.

With the above discussion in mind, fix a top dimensional cell $\cH_\Theta$ in the tropical Hurwitz space. We need to understand the dilation factor that this map induces on integral structures. Recall that the coordinates on the cone $\cM_{\Gamma_{src}}$ are given by $val(\xi_i)$, the valuation of the deformation parameters. Recall from our deformation theory computations in Section~\ref{sec: def2},
\[
br^*(\xi_i) = \widetilde{\xi_i}^{N_i},
\]
where $N_i$ is the LCM of the ramification above the nodes of $D$ lying above the $i$th node of $C$. It follows that
\[
 val(br^*(\xi_i)) = N_i\cdot val(\xi_i).
 \]
 However, as we discussed in Section~\ref{sec: def2}, there are $M = \prod_{e\in E(\Gamma_{tgt})} M_e$ zero strata in the space $\Hbar_{g\to h,d}(\vec\mu)$ for each chosen nodal cover, where $M_e$ is the product of the ramification indices above the node corresponding to $e$, divided by their LCM. Clearly the total contribution $M_eN_e$ at a node corresponding to $e$ is the product of the ramification indices above that node. Ranging over all nodes, we see that the total degree is the product of the edge weights. Keeping in mind Warning~\ref{warning: subdivision}, we see that for each combinatorial type, we recover the weight in the CJM formulae. Appropriately taking into account automorphisms, we recover the desired formula
\[
h_{g\to 0,d}(\vec\mu) = \sum_{\Gamma}\frac{1}{|Aut(\Gamma)|} \prod w(e).
\] 
\qed

\begin{remark}
The induced map $br^{\Sigma}$ on skeleta is an isomorphism on each cone if we forget about the integral structure. The same is true for $br^{trop}$. The enumerative information relies crucially on the integral structure of these cones and the weights, which in turn, relies heavily on the deformation theory.
\end{remark}

\begin{remark}
Although tropicalization is a relatively new to the study of Hurwitz numbers, the spirit of these results is quite classical -- namely, using degeneration techniques to study (enumerative) geometry. In the above computations, we choose a suitable degeneration of the base curve to a rational curve with desirable properties. For instance, the caterpillar curve above (Figure~\ref{fig: m0n-cell}) has the property that it allows us to easily compute local Hurwitz numbers. This strategy was actualized for double Hurwitz numbers in~\cite{CJM2}.
\end{remark}

\subsection{The general correspondence theorem for Hurwitz numbers}\label{subsec-BBM}

We now reprove the general correspondence theorem for Hurwitz numbers from \cite{BBM} at the level of moduli spaces using our newly developed techniques.

Recall from \cite{BBM} that for a tropical admissible cover of combinatorial type $\Theta = [\theta:\Gamma_{src}\to\Gamma_{tgt}]$ as in Section \ref{subsec-tropadm}, the \textit{multiplicity} (depending only on the combinatorial type) is defined to be 
\begin{equation}
\label{multmult}
 \frac{1}{|Aut_0(\Theta)|}\cdot \prod_{v\in \Gamma_{tgt}} H(v) \cdot \prod_{e \in \Gamma_{src} }d_e(\theta),
 \end{equation}
where the second product goes over all interior edges $e$ of $\Gamma_{src}$ and $d_e(\theta)$ denotes their expansion factors.
We decide to mark preimages of branch points, resulting in a simplification in our expression of the local Hurwitz numbers compared to \cite{BBM}.

For a fixed trivalent target tropical curve of genus $h$ with totally degenerate genus function $\Gamma_{tgt}$, in \cite{BBM} the tropical Hurwitz number $h^{trop}_{g\to h,d}(\vec\mu)$ is defined to be the weighted number of admissible covers of $\Gamma_{tgt}$, satisfying the prescribed genus and ramification conditions, counted with the multiplicity defined in \eqref{multmult}. This number does not depend on the choice of $\Gamma_{tgt}$.


\textit{Proof of Theorem~\ref{thm-BBM}.}
From Theorem~\ref{thm:branchdegree}, we know already that $h_{g\to h,d}(\vec\mu)$, which equals the degree of the branch map, also equals the degree of the tropical branch map. All that remains to be seen is that the multiplicity of an admissible cover defined above equals the dilation factors times the weight of the cone of the corresponding combinatorial type. This follows analogously to the proof of the monodromy graph formula as in Section \ref{sec: CJM}. \qed

The monodromy graph formula is implied by Theorem \ref{thm-BBM} using the method of attaching infinite edges in the manner described in Section~\ref{sec: double-numbers}.

\bibliographystyle{siam}
\bibliography{TropicalAdmissibleCovers}

\begin{thebibliography}{10}

\bibitem{ACP}
{\sc D.~Abramovich, L.~Caporaso, and S.~Payne}, {\em The tropicalization of the
  moduli space of curves}, Ann. Sci. {\'E}c. Norm. Sup{\'e}r.,  (To appear).

\bibitem{ACV}
{\sc D.~Abramovich, A.~Corti, and A.~Vistoli}, {\em Twisted bundles and
  admissible covers}, Comm. Algebra, 31 (2003), pp.~3547--3618.

\bibitem{AMR}
{\sc D.~Abramovich, K.~Matsuki, and S.~Rashid}, {\em A note on the
  factorization theorem of toric birational maps after {M}orelli and its
  toroidal extension}, Tohoku Math. J., 51 (1999), pp.~489--537.

\bibitem{ABBR}
{\sc O.~Amini, M.~Baker, E.~Brugall{\'e}, and J.~Rabinoff}, {\em {Lifting
  harmonic morphisms I: metrized complexes and Berkovich skeleta}}, Research in
  the Mathematical Sciences,  (To appear).

\bibitem{BPR}
{\sc M.~Baker, S.~Payne, and J.~Rabinoff}, {\em Nonarchimedean geometry,
  tropicalization, and metrics on curves}, Algebraic Geometry,  (To appear).

\bibitem{Ber90}
{\sc V.~G. Berkovich}, {\em Spectral theory and analytic geometry over
  non-Archimedean fields}, vol.~33, American Mathematical Society, 1990.

\bibitem{BBM}
{\sc B.~Bertrand, B.~Brugall{\'e}, and G.~Mikhalkin}, {\em Tropical open
  {H}urwitz numbers}, Rend. Semin. Mat. Univ. Padova, 125 (2011), pp.~157--171.

\bibitem{BMV11}
{\sc S.~Brannetti, M.~Melo, and F.~Viviani}, {\em On the tropical {T}orelli
  map}, Adv. Math., 226 (2011), pp.~2546--2586.

\bibitem{BM13}
{\sc A.~Buchholz and H.~Markwig}, {\em Tropical covers of curves and their
  moduli spaces}, Comm. Contemp. Math., 17 (2015), p.~1350045.

\bibitem{Cap11}
{\sc L.~Caporaso}, {\em Algebraic and tropical curves: comparing their moduli
  spaces}, To appear in Handbook of Moduli, edited by G. Farkas and I.
  Morrison,  (2011).

\bibitem{cap:gonality}
\leavevmode\vrule height 2pt depth -1.6pt width 23pt, {\em Gonality of
  algebraic curves and graphs}, in Algebraic and complex geometry, vol.~71 of
  Springer Proc. Math. Stat., Springer, Cham, 2014, pp.~77--108.

\bibitem{CJM1}
{\sc R.~Cavalieri, P.~Johnson, and H.~Markwig}, {\em Tropical {H}urwitz
  numbers}, J. Alg. Combin., 32 (2010), pp.~241--265.

\bibitem{CJM2}
\leavevmode\vrule height 2pt depth -1.6pt width 23pt, {\em Wall crossings for
  double {H}urwitz numbers}, Adv. Math., 228 (2011), pp.~1894--1937.

\bibitem{renzosbook}
{\sc R.~Cavalieri and E.~Miles}, {\em From Riemann Surfaces to Algebraic
  Geometry: A First Course in Hurwitz Theory}, Cambridge University Press,
  2015.

\bibitem{CMV12}
{\sc M.~Chan, M.~Melo, and F.~Viviani}, {\em {Tropical Teichm\"uller and Siegel
  spaces}}, Proceedings of the CIEM workshop in tropical Geometry, Contemporary
  Mathematics, 589 (2013), pp.~45--85.

\bibitem{ELSV}
{\sc T.~Ekedahl, S.~Lando, M.~Shapiro, and A.~Vainshtein}, {\em Hurwitz numbers
  and intersections on moduli spaces of curves}, Invent. Math., 146 (2001),
  pp.~297--327.

\bibitem{GKM07}
{\sc A.~Gathmann, M.~Kerber, and H.~Markwig}, {\em Tropical fans and the moduli
  space of rational tropical curves}, Compos. Math., 145 (2009), pp.~173--195.

\bibitem{GJ1}
{\sc I.~P. Goulden and D.~M. Jackson}, {\em Transitive factorisations into
  transpositions and holomorphic mappings on the sphere}, Proc. Amer. Math.
  Soc., 125 (1997), pp.~51--60.

\bibitem{GJ2}
\leavevmode\vrule height 2pt depth -1.6pt width 23pt, {\em Transitive powers of
  {Y}oung-{J}ucys-{M}urphy elements are central}, J. Algebra, 321 (2009),
  pp.~1826--1835.

\bibitem{GV05}
{\sc T.~Graber and R.~Vakil}, {\em Relative virtual localization and vanishing
  of tautological classes on moduli spaces of curves}, Duke Math. J., 130
  (2005), pp.~1--37.

\bibitem{GS13}
{\sc M.~Gross and B.~Siebert}, {\em {Logarithmic Gromov-Witten invariants}}, J.
  Amer. Math. Soc., 26 (2013), pp.~451--510.

\bibitem{HM82}
{\sc J.~Harris and D.~Mumford}, {\em {On the Kodaira dimension of the moduli
  space of curves}}, Invent. Math., 67 (1982), pp.~23--86.

\bibitem{KKMSD}
{\sc G.~Kempf, F.~Knudsen, D.~Mumford, and B.~Saint-Donat}, {\em Toroidal
  embeddings {I}}, Lecture Notes in Mathematics, 339 (1973).

\bibitem{LZ99}
{\sc S.~Lando and D.~Zvonkine}, {\em On multiplicities of the
  {L}yashko-{L}ooijenga mapping on discriminant strata}, Funct. Anal. Appl., 33
  (1999), pp.~178--188.

\bibitem{Mi03}
{\sc G.~Mikhalkin}, {\em Enumerative tropical geometry in {${\mathbb{R}^2}$}},
  J. Amer. Math. Soc, 18 (2005), pp.~313--377.

\bibitem{Moch}
{\sc S.~Mochizuki}, {\em The geometry of the compactification of the {H}urwitz
  scheme}, PhD thesis, Princeton University, 1992.

\bibitem{OP06}
{\sc A.~Okounkov and R.~Pandharipande}, {\em Gromov-{W}itten theory, {H}urwitz
  theory, and completed cycles}, Ann. of Math., 163 (2006), pp.~517--560.

\bibitem{OP}
\leavevmode\vrule height 2pt depth -1.6pt width 23pt, {\em Gromov-{W}itten
  theory, {H}urwitz numbers, and matrix models}, in Algebraic
  geometry---{S}eattle 2005. {P}art 1, vol.~80 of Proceedings of Symposia in
  Pure Mathematics, American Mathematical Society, Providence, RI, 2009,
  pp.~325--414.

\bibitem{R12}
{\sc J.~Rabinoff}, {\em Tropical analytic geometry, {N}ewton polygons, and
  tropical intersections}, Adv. Math., 229 (2012), pp.~3192--3255.

\bibitem{Thu07}
{\sc A.~Thuillier}, {\em {G{\'e}om{\'e}trie toro{\"\i}dale et g{\'e}om{\'e}trie
  analytique non archim{\'e}dienne. Application au type d'homotopie de certains
  sch{\'e}mas formels}}, Manuscripta Math., 123 (2007), pp.~381--451.

\bibitem{U13}
{\sc M.~Ulirsch}, {\em Functorial tropicalization of logarithmic schemes: The
  case of constant coefficients}, arXiv:1310.6269,  (2013).

\end{thebibliography}
\end{document}